\documentclass[12pt]{article}
\usepackage{amsmath}
\usepackage{latexsym}
\usepackage{amssymb}
\newtheorem{thm}{Theorem}[section]
\newtheorem{la}[thm]{Lemma}
\newtheorem{Defn}[thm]{Definition}
\newtheorem{Remark}[thm]{Remark}
\newtheorem{Note}[thm]{Note}
\newtheorem{prop}[thm]{Proposition}

\newtheorem{cor}[thm]{Corollary}
\newtheorem{Example}[thm]{Example}
\newtheorem{Examples}[thm]{Examples}
\newtheorem{Problems}[thm]{Problems}

\newtheorem{Problem}[thm]{Problem}
\newtheorem{Convention}[thm]{Convention}
\newtheorem{Number}[thm]{\!\!}
\newenvironment{defn}{\begin{Defn}\rm}{\end{Defn}}

\newenvironment{rem}{\begin{Remark}\rm}{\end{Remark}}
\newenvironment{numba}{\begin{Number}\rm}{\end{Number}}
\newenvironment{proof}{{\noindent\bf Proof.}}%
                  {\nopagebreak\hspace*{\fill}$\Box$\medskip\medskip\par}   
\newcommand{\Punkt}{\nopagebreak\hspace*{\fill}$\Box$}
\newcommand{\wb}{\overline}
\newcommand{\ve}{\varepsilon}
\newcommand{\at}{\symbol{'100}}

\newcommand{\wt}{\widetilde}

\newcommand{\impl}{\Rightarrow}

\newcommand{\mto}{\mapsto}

\newcommand{\isom}{\cong}

\newcommand{\N}{{\mathbb N}}
\newcommand{\R}{{\mathbb R}}

\newcommand{\cG}{{\mathcal G}}

\newcommand{\C}{{\mathbb C}}

\newcommand{\cU}{{\cal U}}
\newcommand{\cO}{{\cal O}}

\newcommand{\cg}{{\mathfrak g}}

\newcommand{\dl}{{\displaystyle \lim_{\longrightarrow}}}

\newcommand{\one}{{\bf 1}}
\newcommand{\sub}{\subseteq}

\DeclareMathOperator{\im}{im}

\DeclareMathOperator{\id}{id}

\newcommand{\cB}{{\cal B}}

\newcommand{\cF}{{\cal F}}

\newcommand{\cK}{{\cal K}}
\newcommand{\cR}{{\cal R}}
\newcommand{\cL}{{\cal L}}
\newcommand{\cT}{{\cal T}}

\DeclareMathOperator{\conv}{conv}

\DeclareMathOperator{\Supp}{supp}

\DeclareMathOperator{\Mea}{M}

\DeclareMathOperator{\op}{op}

\begin{document}
$\;$\\[-27mm]
\begin{center}
{\Large\bf Continuity of convolution of test functions\\[2mm]
on Lie groups}\\[7mm]
{\bf Lidia Birth and Helge Gl\"{o}ckner}\vspace{4mm}
\end{center}
\begin{abstract}\noindent
For a Lie group $G$, we show that the map
$C^\infty_c(G)\times C^\infty_c(G)\to C^\infty_c(G)$,
$(\gamma,\eta)\mapsto \gamma*\eta$
taking a pair of
test functions to their convolution is continuous if and only if $G$ is $\sigma$-compact.
More generally, consider $r,s,t
\in \N_0\cup\{\infty\}$ with $t\leq r+s$, locally convex spaces $E_1$, $E_2$
and a continuous bilinear map $b\colon E_1\times E_2\to F$
to a complete locally convex space $F$.
Let $\beta\colon C^r_c(G,E_1)\times C^s_c(G,E_2)\to C^t_c(G,F)$,
$(\gamma,\eta)\mto \gamma *_b\eta$ be the associated convolution map.
The main result is a characterization of those $(G,r,s,t,b)$
for which $\beta$ is continuous.
Convolution
of compactly supported continuous functions on a locally compact group
is also discussed, as well as convolution of compactly supported $L^1$-functions
and convolution of compactly supported Radon measures.
\vspace{-.5mm}
\end{abstract}
{\footnotesize {\em Classification}:
22E30,
%
46F05 (Primary);
%
22D15,
%
42A85,
%
43A10,
%
43A15,
%
46A03,
%
%
46E25.\\[1mm]
%
%
}\\[6mm]
\noindent
{\bf Introduction and statement of results}\\[2.5mm]
It has been known since the beginnings of distribution theory
that the bilinear convolution map $\beta\colon C^\infty_c(\R^n)\times C^\infty_c(\R^n)\to C^\infty_c(\R^n)$,
$(\gamma,\eta)\mto \gamma*\eta$
(and even convolution $C^\infty(\R^n)'\times C^\infty_c(\R^n)\to C^\infty_c(\R^n)$)
is hypocontinuous~\cite[p.\,167]{Sw}.
However, a proof for continuity of $\beta$ was only published
recently \cite[Proposition 2.3]{Hir}.
The second author gave an alternative proof~\cite{Glo}, which
is based on a
continuity criterion
for bilinear mappings on locally convex direct sums.
Our goal is to adapt the latter method to the case where $\R^n$
is replaced with a Lie group, and to the convolution of vector-valued functions.\\[2.5mm]
Let $b\colon E_1\times E_2\to F$ be a continuous bilinear map
between locally convex spaces such that $b\not=0$.
Let $r,s,t\in \N_0\cup\{\infty\}$
with $t\leq r+s$.
If $r=s=t=0$, let $G$ be a locally compact group;
otherwise, let $G$ be a Lie group.
Let $\lambda_G$ be a left Haar measure on $G$.
If $G$ is discrete, we need not impose any completeness assumptions on $F$.
If $G$ is metrizable and not discrete, we assume that $F$ is sequentially complete
or satisfies the metric convex compactness property
(i.e., every metrizable compact subset of~$F$ has a relatively compact convex hull).
If $G$ is not metrizable (and hence not discrete),
we assume that $F$ satisfies the convex compactness property
(i.e., every compact subset of $F$ has a relatively compact convex hull);
this is guaranteed if $F$ is quasi-complete.\footnote{See \cite{Voi}
for a discussion of these properties.}
These conditions ensure the existence of the integrals needed to define the convolution
$\gamma*_b\eta\colon G\to F$
of $\gamma\in C^r_c(G,E_1)$ and $\eta\in C^s_c(G,E_2)$~via
\begin{equation}\label{defco}
(\gamma*_b\eta)(x):=\int_G b(\gamma(y), \eta(y^{-1}x))\, d\lambda_G(y)\qquad \mbox{for $\, x\in G$.}
\end{equation}
Then $\gamma*_b\eta\in C^{r+s}_c(G,F)$ (Proposition~\ref{diffbar}), enabling us to consider the map
\begin{equation}\label{defbeta}
\beta\colon
C^r_c(G, E_1)\times C^s_c(G, E_2)\to C^t_c(G, F)\,,\quad
(\gamma,\eta)\mapsto \gamma*_b\eta\, .
\end{equation}
The mapping $\beta$ is bilinear, and it is always
hypo\-continuous (Proposition~\ref{hypo}).
If $G$ is compact, then $\beta$
is continuous (Corollary~\ref{cpcase}).
If $G$ is an infinite discrete group, then $\beta$ is continuous if and only if $G$ is countable
and $b$ `admits product estimates' (Proposition~\ref{discrca}), in the following sense:\\[4mm]
{\bf Definition.} Let $b\colon E_1\times E_2\to F$ be a continuous bilinear map
between locally convex spaces. We say that $b$ \emph{admits product estimates}
if, for each double sequence $(p_{i,j})_{i,j\in \N}$
of continuous seminorms on~$F$,
there exists a sequence $(p_i)_{i\in \N}$ of continuous seminorms on $E_1$
and a sequence $(q_j)_{j\in\N}$ of continuous seminorms on $E_2$,
such that
\begin{equation}\label{prodest}
(\forall i,j\in \N)\,(\forall x\in E_1)\,(\forall y\in E_2)\quad
p_{i,j}(b(x,y))\leq p_i(x)q_j(y)\,.\vspace{2mm}
\end{equation}
Having dealt with compact groups and discrete groups,
only one case remains:\\[4mm]
{\bf Theorem A.} \emph{If $G$ is neither discrete nor compact,
then the convolution map $\beta$ from} (\ref{defbeta}) \emph{is
continuous if and only if all of}
(a), (b) \emph{and} (c) \emph{are satisfied}:
\begin{itemize}
\item[(a)]
\emph{$G$ is $\sigma$-compact};
\item[(b)]
\emph{If $t=\infty$, then also $r=s=\infty$};
\item[(c)]
\emph{$b$ admits product estimates.}\vspace{2mm}
\end{itemize}
We mention that (c) is automatically satisfied whenever both $E_1$ and $E_2$
are normable \cite[Corollary~4.2]{PRA}.
As a consequence, for normable $E_1$, $E_2$ and a Lie group $G$,
the convolution map $\beta\colon C^\infty_c(G,E_1)\times C^\infty_c(G,E_2)\to C^\infty_c(G,F)$
is continuous if and only if~$G$ is $\sigma$-compact.
In particular, the convolution map $C^\infty_c(G)\times C^\infty_c(G)\to C^\infty_c(G)$
is continuous for each $\sigma$-compact Lie group~$G$
(as first established in the unpublished thesis \cite{Bir},
by a different reasoning), but fails to be continuous if $G$ is not $\sigma$-compact.\\[2.3mm]
Further examples of bilinear maps admitting product estimates can be found in~\cite{PRA}.
For instance, the convolution map $C^\infty(G)\times C^\infty(G)\to C^\infty(G)$
admits product estimates whenever $G$ is a compact Lie group.
Of course, not every continuous bilinear map does admit product estimates,
e.g., the multiplication map $C^\infty[0,1]\times C^\infty[0,1]\to C^\infty[0,1]$
\cite[Example~5.2]{PRA}.
In particular, this gives us an example of a topological algebra~$A$
such that the associated convolution map
$C^\infty_c(\R,A)\times C^\infty_c(\R,A)\to C^\infty_c(\R,A)$
is discontinuous.
It is also interesting that
the convolution map $C^\infty_c(\R)\times C^0_c(\R)\to C^\infty_c(\R)$
is discontinuous (as condition~(b) from Theorem~A is violated here).
This had not been recorded yet
in the works \cite{Hir} and \cite{Glo} devoted to $G=\R^n$.\\[3mm]
Irrespective of locally compactness,
we have some information concerning
convolution on the space
$\Mea_c(G)=\dl\, \Mea_K(G)$\vspace{-.3mm}
of compactly supported complex Radon measures on a Hausdorff topological group~$G$.
Recall that a topological space~$X$ is called \emph{hemicompact}
if $X=\bigcup_{n=1}^\infty K_n$ with compact subsets $K_1\sub K_2\sub\cdots$
of $X$, such that each compact subset $K\sub X$ is contained in some $K_n$.
A locally compact space is hemicompact if and only if it is $\sigma$-compact.
We call a Hausdorff topological group~$G$ \emph{spacious}
if there exist uncountable subsets $A,B \sub G$
such that $\{(x,y)\in A\times B\colon xy\in K\}$ is finite
for each compact subset $K\sub G$.
A locally compact group is spacious if and only if it is not $\sigma$-compact (see Remark~\ref{remspa}).\\[4mm]
{\bf Theorem B.} \emph{Let $G$ be a Hausdorff group
and $\beta\colon \Mea_c(G)\times \Mea_c(G)\to \Mea_c(G)$, $(\mu,\nu)\mto \mu * \nu$
be the convolution map.}
\begin{itemize}
\item[(a)]
\emph{If $G$ is hemicompact, then $\beta$ is continuous.}
\item[(b)]
\emph{If $G$ is spacious, then $\beta$ is not continuous.}
\end{itemize}
Thus, for locally compact $G$, the convolution map $\beta$ from Theorem~B\linebreak
is continuous
if and only if $G$ is $\sigma$-compact.
An analogous conclusion\linebreak
applies to convolution of compactly supported
$L^1$-functions on a locally compact group (Corollary~\ref{L1c}).
Hemicompact groups arise in the duality theory\linebreak
of
abelian topological groups,
because dual groups of abelian metrizable groups are hemicompact,
and dual groups of abelian hemicompact groups are\linebreak
metrizable
(\cite{Au1}; see \cite{Cha},  \cite{Au2}, \cite{Au3}, \cite{GGH} for recent studies of such groups).\\[2.5mm]
We also discuss the convolution map $C^r_c(G,E_1)\times C^s(G,E_2)\to C^t(G, F)$.
It is hypocontinuous,
but continuous only if $G$ is compact (Proposition~\ref{nospp}).
As a consequence,
neither the action $C^\infty_c(G)\times E\to E$
(nor the action $C^\infty_c(G) \times E^\infty \to E^\infty$ on the space of smooth vectors)
associated to a continuous action $G\times E\to E$
of a Lie group $G$ on a Fr\'{e}chet space~$E$
need to be continuous (contrary to a claim recently made \cite[pp.\,667--668]{Kro}).
In fact, if $G$ is $\R$ and $\R\times C^\infty(\R)\to C^\infty(\R)$
the translation action, then $C^\infty_c(\R)$ acts on $C^\infty(\R)$
by the convolution map, which is discontinuous by Proposition~\ref{nospp}
(or the independent study \cite{Lar}).
For details, we refer the reader to \cite[Proposition~A]{GLF}.\\[2.5mm]
The $(G,r,s,t,b)$ for which $\beta$ admits product
estimates are also known~\cite{PRA}.\\[2.5mm]
For recent studies of convolution of vector-valued distributions,
we refer to \cite{Ba1}, \cite{Ba2} and the references therein.
Larcher~\cite{Lar} gives a systematic account of the
continuity properties of convolution between classical spaces of
scalar-valued functions and distributions on $\R^n$,
and proves discontinuity in some cases
in which convolution was previously considered continuous by some authors
(like \cite{Ehr} and \cite{Shi}).\\[3mm]
{\bf Structure of the article.}
Sections~\ref{secprel} to~\ref{secmeasset} are of a preparatory nature
and provide basic notation and facts
which are similar to familiar special cases
and easy to take on faith.
Because no direct references are available
in the required generality, we do not omit the proofs
(which follow classical ideas),
but relegate them to an appendix (Appendix~\ref{appproofs}).
Appendices~A and~B compile further preliminaries concerning
vector-valued integrals and hypocontinuous bilinear maps.
On this footing, our results
are established in Sections~\ref{aness} to~\ref{secfin}.\\[2.6mm]
\emph{Acknowledgement.} The continuity question concerning convolution on $C^\infty_c(G)$
was posed to the second author by Karl-Hermann Neeb (Erlangen) in July 2010
(and arose in a research project by the latter and Gestur Olafsson,
cf.\ \cite[Proposition~2.8]{NaO}).
The research was supported by DFG, grant GL 357/5--2.
\section{Preliminaries and notation}\label{secprel}
In this section, we compile notation and basic facts concerning
spaces of vector-valued $C^r$-functions.
The proofs are given in Appendix~\ref{appproofs}.\\[2.3mm]
{\bf Basic conventions.} We write $\N=\{1,2,\ldots\}$ and $\N_0:=\N\cup\{0\}$.
By a \emph{locally convex space}, we mean a Hausdorff locally convex real
topological vector space. If $E$ is such a space, we write $E'$ for the space of continuous\linebreak
linear functionals on~$E$.
A map between topological spaces
is called a \emph{topological embedding} if it is a homeomorphism
onto its image. If $E$ is vector space
and $p$ a seminorm on~$E$, we define
$B^p_r(x):=\{y\in E\colon p(y-x)<r\}$
and $\wb{B}^p_r(x):=\{y\in E\colon p(y-x)\leq r\}$
for $x\in E$ and $r>0$.
If $X$ is a set and $\gamma\colon X\to E$
a map,
we let $\|\gamma\|_{p,\infty}:=\sup_{x\in X}p(\gamma(x))$.
If $(E,\|.\|)$ is a normed space and $p=\|.\|$,
we write $\|\gamma\|_\infty$ instead of $\|\gamma\|_{p,\infty}$,
and $B^E_r(x)$ instead of $B^p_r(x)$.
Apart from $\rho\,d\mu$, we shall also write
$\rho\odot\mu$ for measures with a\linebreak
density.
The manifolds considered in this article are finite-dimensional,
but not necessarily $\sigma$-compact or paracompact (unless the contrary is stated).
The Lie groups considered are finite-dimensional, real Lie groups.\\[2.7mm]
{\bf Vector-valued {\boldmath$C^r$}-functions.}
Let $E$ and $F$ be locally convex spaces, $U\sub E$ an open set
and $r\in \N_0\cup\{\infty\}$.
Then a map $\gamma \colon U\to F$ is called $C^r$ if it is continuous,
the iterated directional derivatives $d^{(j)}\gamma(x,y_1,\ldots, y_j):=(D_{y_j}\cdots D_{y_1}\gamma)(x)$
exist for all $j\in \N$ such that $j\leq r$, $x\in U$ and $y_1,\ldots, y_j\in E$,
and, moreover, each of the maps $d^{(j)}\gamma \colon U\times E^j\to F$ is continuous.
See \cite{Mic}, \cite{Mil}, \cite{Ham},
\cite{RES}, \cite{GaN}
for the theory of such functions (in varying degrees of generality
as regards $E$ and $F$).
If $E=\R^n$, then a vector-valued function~$\gamma$ as before is $C^r$ if and only if the partial derivatives
$\partial^\alpha \gamma\colon U\to F$ exist and are continuous,
for all multi-indices $\alpha=(\alpha_1,\ldots,\alpha_n)\in \N_0^n$ such that $|\alpha|:=\alpha_1+\cdots+\alpha_n\leq r$
(see, e.g., \cite{GaN}).
Since compositions of $C^r$-maps are $C^r$,
it makes sense to consider $C^r$-maps from
$C^r$-manifolds to locally convex spaces.
If $M$ is a $C^1$-manifold and
$\gamma\colon M\to E$ a $C^1$-map to a locally convex space,
we write $d\gamma$ for the second component
of the tangent map $T\gamma\colon TM\to TE\isom E\times E$.
If $X$ is a vector field on~$M$, we define
\begin{equation}\label{defnDX}
D_X(\gamma):=X.\gamma:=d\gamma\circ X\,.
\end{equation}
{\bf Function spaces and their topologies.}
Let $r\in \N_0\cup\{\infty\}$ now and $E$ be a locally convex space.
If $r=0$, let
$M$ be a (Hausdorff) locally compact space,
and equip the space $C^0(M,E):=C(M,E)$ of continuous $E$-valued functions on~$M$
with the compact-open topology
given by the seminorms
\[
\|.\|_{p,K}\colon C(M,E)\to [0,\infty[\,,\quad \gamma\mto \|\gamma|_K\|_{p,\infty},
\]
for $K$ ranging through the compact subsets of~$M$,
and $p$ through the continuous seminorms on~$E$.
If $(E,\|.\|_E)$ is a normed space, we abbreviate $\|.\|_K:=\|.\|_{\|.\|_E,K}$.
To harmonize notation, write
$T^0M:=M$ and $d^0\gamma:=\gamma$ for $\gamma\in C^0(M,E):=C(M,E)$.
If $r>0$, let $M$ be a $C^r$-manifold.
For $k\in\N$ with $k\leq r$, set $T^kM:=T(T^{k-1}M)$ and $d^k\gamma:=d(d^{k-1}\gamma)\colon T^kM\to E$
for $C^k$-maps $\gamma\colon M\to E$. Thus $T^1M=TM$ and $d^1\gamma=d\gamma$.
Equip $C^r(M,E)$ with the initial topology with respect to the maps
$d^k\colon C^r(M,E)\to C(T^kM, E)$ for $k\in \N_0$ with $k\leq r$,
where $C(T^k(M),E)$ is equipped with the compact-open topology.
Returning to $r\in \N_0\cup\{\infty\}$,
endow $C^r_A(M,E):=\{\gamma\in C^r(M,E)\colon \text{supp}(\gamma)\sub A\}$
with the topology induced by $C^r(M,E)$,
for each closed subset $A\sub M$.
Let $\cK(M)$ be the set of compact subsets of~$M$.
Give $C^r_c(M,E):=\bigcup_{K\in \cK(M)}\, C^r_K(M,E)$
the locally convex direct limit topology.
Since each inclusion map
$C^r_K(M,E)\to C^r(M,E)$ is continuous and linear,
also the linear inclusion map $C^r_c(M,E)\to C^r(M,E)$ is continuous.
Since $C^r(M,E)$ is Hausdorff, this implies that also $C^r_c(M,E)$ is Hausdorff.
We\linebreak
abbreviate $C^r(M):=C^r(M,\R)$,
$C^r_K(M):=C^r_K(\R)$ and $C^r_c(M):=C^r_c(M,\R)$.\\[2.7mm]
{\bf Facts concerning direct sums.}
If $(E_i)_{i\in I}$ is a family of locally
convex spaces,
we shall always equip the direct sum $E:=\bigoplus_{i\in I}E_i$
with the locally convex direct sum topology~\cite{Bou}.
We often identify $E_i$ with its
image in~$E$.
\begin{rem}\label{semisum}
If $U_i\sub E_i$ is a $0$-neighbourhood for $i\in I$,
then the convex hull
$U:=\conv\big(\bigcup_{i\in I} U_i\big)$
is a $0$-neighbourhood in~$E$,
and a basis of $0$-neighbourhoods is obtained in this way
(as is well-known).
If $I$ is countable, then the corresponding `boxes'
$\bigoplus_{i\in I}U_i:=E\cap \prod_{i\in I}U_i$ form
a basis of $0$-neighbourhoods in~$E$ (cf.\ \cite{Jar}).
It is clear from this that the topology on~$E$ is defined
by the seminorms $q\colon E\to [0,\infty[$ taking $x=(x_i)_{i\in I}$
to $\sum_{i\in I}q_i(x_i)$, for $q_i$ ranging through the set of continuous
seminorms on~$E_i$
(because $B^q_1(0)=\conv(\bigcup_{i\in I}B^{q_i}_1(0))$.)
If $I$ is countable, we can take the seminorms $q(x):=\max\{q_i(x_i)\colon i\in I\}$
instead (because $B^q_1(0)=\bigoplus_{i\in I}B^{q_i}_1(0)$.)
\end{rem}
\begin{la}\label{sumemb}
Let $(E_i)_{i\in I}$ and $(F_i)_{i\in I}$ be families of locally
convex spaces and $\lambda_i\colon E_i\to F_i$ be
a linear map that is topological embedding,
for $i\in I$.
Then $\oplus_{i\in I}\lambda_i\colon \bigoplus_{i\in I}E_i\to\bigoplus_{i\in I}F_i$,
$(x_i)_{i\in I} \mto (\lambda_i(x_i))_{i\in I}$
is a topological embedding.
\end{la}
{\bf Mappings to direct sums.}
\begin{la}\label{lcsum}
Let $r\in \N_0\cup\{\infty\}$.
If $r=0$, let $M$ be a locally compact space.
If $r>0$, let
$M$ be a $C^r$-manifold.
Let~$E$ be a locally convex space,
and $(h_j)_{j\in J}$ be a family of
functions $h_j\in C^r_c(M)$
whose supports $K_j:=\Supp(h_j)$
form a locally finite family.
Then the map
\[
\Phi\colon C^r_c(M,E)\to\bigoplus_{j\in J}C^r_{K_j}(M,E)
\,,\quad \gamma\mapsto (h_j\cdot \gamma)_{j\in J}
\]
is continuous and linear. If $(h_j)_{j\in J}$ is a partition of unity $($i.e., $h_j\geq 0$ and
$\sum_{j\in J}h_j=1$ pointwise$)$, then~$\Phi$ is a topological embedding.
\end{la}
\begin{la}\label{SJ}
Let $r\in \N_0\cup\{\infty\}$.
If $r=0$, let $M$ be a locally compact space.
If $r>0$, let
$M$ be a $C^r$-manifold.
Let~$E$ be a locally convex space,
and~$P$ be a set of disjoint, open and closed subsets
of~$M$, such that $(S)_{S\in P}$ is locally finite.
Then\vspace{-3mm}
\[
\Phi\colon C^r_c(M,E)\to\bigoplus_{S\in P}C^r_c(S,E)
\,,\quad \gamma\mapsto (\gamma|_S)_{S\in P}\vspace{-2mm}
\]
is a continuous linear map.
If~$P$ is a partition of~$M$ into open sets,
then~$\Phi$ is an isomorphism of topological vector spaces.
\end{la}
{\bf Seminorms arising from frames.}
If $M$ is a smooth manifold of dimension $m$,
we call a set $\cF=\{X_1,\ldots, X_m\}$ of smooth vector fields
a \emph{frame} on~$M$ if $X_1(p),\ldots, X_m(p)$ is a basis
for $T_p(M)$, for each $p\in M$.
If also $\cG=\{Y_1,\ldots, Y_m\}$ is a frame
on $M$, then there exist $a_{i,j}\in C^\infty(M)$ for $i,j\in\{1,\ldots, m\}$
such that $Y_j=\sum_{i=1}^m a_{i,j} X_i$.
\begin{la}\label{framediff}
Let $M$ be a smooth manifold,
$E$ be a locally convex space, $k, \ell \in \N_0$
and $\cF_1,\ldots, \cF_k$ be frames on~$M$.
Let $\gamma \colon M\to E$ be a $C^k$-map
such that $X_j\cdots X_1.\gamma\in C^\ell(M,E)$
for all $j\in \N_0$ with $j\leq k$ and $X_i\in \cF_i$ for $i\in \{1,\ldots,j\}$.
Then $\gamma$ is $C^{k+\ell}$.
\end{la}
\begin{la}\label{topframe}
Let $E$ be a locally convex space,
$M$ be a smooth manifold,
$r \in \N$ and $\cF:=(\cF_1,\ldots, \cF_r)$
be an $r$-tuple of frames on~$M$.
Then the usual topology $\cO$ on $C^r(M,E)$
coincides with the initial topology $\cT_\cF$ with respect to the maps
\[
D_{X_j,\ldots, X_1}\colon C^r(M,E)\to C^0(M,E)_{c.o.}, \quad \gamma \mto X_j\ldots X_1.\gamma \, ,
\]
where $j\in \{0,\ldots, r\}$ and $X_i\in \cF_i$ for $i\in \{1,\ldots, j\}$.
As a consequence, for each closed subset $K\sub M$,
the topology on $C^r_K(M,E)$ is initial with respect to the maps
$C^r_K(M,E)\to C^0_K(M,E)_{c.o.}$,
$\gamma \mto X_j\ldots X_1.\gamma$,
where $j\in \{0,\ldots, r\}$ and $X_i\in \cF_i$ for $i\in \{1,\ldots, j\}$.
\end{la}
\begin{defn}\label{deflrmrtr}
Let $G$ be a Lie group, with identity element $1$.
Given $g\in G$,
we define the left translation map
$L_g\colon G\to G$, $L_g(x):=gx$
and the right translation map
$R_g\colon G\to G$, $R_g(x):=xg$.
Let $\cB$ be a basis of the tangent space $T_1(G)$,
and $E$ be a locally convex space.
For $v\in\cB$, let $\cL_v$ be the left-invariant vector field
on $G$ defined via $\cL_v(g):=T_1(L_g)(v)$,
and $\cR_v$ the right-invariant vector field given by $\cR_v(g):=T_1(R_g)(v)$. 
Write
\[
\cF_L:=\{\cL_v\colon v\in \cB \}\quad\mbox{and}\quad \cF_R:=\{\cR_v \colon v\in \cB \}\,.
\]
Let $K\sub G$ be compact.
Given $r\in \N_0\cup\{\infty\}$, $k,\ell\in\N_0$ with $k+\ell \leq r$, and a continuous seminorm~$p$ on~$E$,
we define $\|\gamma\|^L_{k,p}$ (resp., $\|\gamma\|^R_{k,p}$) for $\gamma\in C^r_K(G,E)$
as the maximum of the numbers
\[
\|X_j\ldots X_1.\gamma\|_{p,\infty}\,,
\]
for $j\in \{0,\ldots, k\}$ and $X_1,\ldots, X_j\in \cF_L$
(resp., $X_1,\ldots, X_j\in \cF_R$).
We also define
$\|\gamma\|^{L,R}_{k,\ell,p}$ (resp., $\|\gamma\|^{R,L}_{k,\ell,p}$)
as the maximum of the numbers
\[
\|X_i\ldots X_1.Y_j\ldots Y_1.\gamma\|_{p,\infty}\,,
\]
for $i\in \{0,\ldots, k\}$, $j\in \{0,\ldots, \ell\}$ and
$X_1,\ldots, X_i\in \cF_L$,
$Y_1,\ldots, Y_j\in \cF_R$
(resp., $X_1,\ldots, X_i\in \cF_R$ and
$Y_1,\ldots, Y_j\in \cF_L$).
Then $\|.\|^L_{k,p}$, $\|.\|^R_{k,p}$, $\|.\|^{L,R}_{k,\ell,p}$
and $\|.\|^{R,L}_{k,\ell,p}$
are seminorms on $C_K^r(G,E)$.
If $E=\R$ and $p=|.|$,
we relax notation and also write
$\|.\|^L_k$, $\|.\|^R_k$, $\|.\|^{L,R}_{k,\ell}$ and $\|.\|^{R,L}_{k,\ell}$
instead of
$\|.\|^L_{k,p}$, $\|.\|^R_{k,p}$, $\|.\|^{L,R}_{k,\ell,p}$ and $\|.\|^{R,L}_{k,\ell,p}$,
respectively. The same symbols will be used for the
corresponding seminorms on $C^r_c(G,E)$ (defined
by the same formulas). For $\ell\in \N_0$ with $\ell\leq r$,
we shall also need the seminorm
$\|.\|_{\ell,K}^L$ on $C^r_c(G)$ defined
as the maximum of the numbers
$\|X_j\ldots X_1.\gamma|_K\|_\infty$
for $j\in \{0,\ldots, \ell\}$ and $X_1,\ldots, X_j\in \cF_L$.
For each compact set $A\sub G$, we have
$\|\gamma\|^L_{\ell,K}\leq \|\gamma\|^L_\ell$
for each $\gamma\in C^r_A(G)$.
Hence $\|.\|^L_{\ell,K}$ is continuous on $C^r_A(G)$
for each~$A$ and hence continuous
on the locally convex direct limit $C^r_c(G)$.
\end{defn}
To enable uniform notation in the proofs for Lie groups and locally compact groups,
we shall write $\|.\|^L_{0,p}:=\|.\|^R_{0,p}:=\|^{R,L}_{0,0,p}:=\|.\|^{L,R}_{0,0,p}:=\|.\|_{p,\infty}$
if $p$ is a continuous seminorm on~$E$ and $G$ a locally compact
group. If $E=\R$ and $K\sub G$ is a compact set, we shall
also write $\|.\|^L_{0,K}:=\|.\|_K$.\\[3mm]
In the situation of Definition~\ref{deflrmrtr}, we have:
\begin{la}\label{enoughsn}
For each $t\in\N_0\cup\{\infty\}$, compact set $K\sub G$
and locally convex space~$E$, the topology on $C^t_K(G,E)$
coincides with the topologies defined by each of the following families
of seminorms:
\begin{itemize}
\item[\rm (a)]
The family of the seminorms $\|.\|^L_{j,p}$, for $j\in \N_0$ such that $j\leq t$
and continuous seminorms~$p$ on~$E$;
\item[\rm (b)]
The family of the seminorms $\|.\|^R_{j,p}$, for $j\in \N_0$ such that $j\leq t$
and continuous seminorms~$p$ on~$E$.
\end{itemize}
If $t<\infty$ and $t=k+\ell$, then the topology on $C^t_K(G,E)$
is also defined by the seminorms $\|.\|^{L,R}_{k,\ell,p}$, for continuous
seminorms $p$ on $E$ $($respectively,
by the seminorms $\|.\|^{R,L}_{k,\ell,p})$.
\end{la}
{\bf Useful automorphisms.}
We record several isomorphisms
of topological vector spaces,
for later use.
\begin{defn}
If $G$ is a group, $\gamma\colon G\to E$
a map to a vector space and $g\in G$,
we define the left translate
$\tau^L_g(\gamma)\colon G\to E$ and the right translate
$\tau^R_g(\gamma)\colon G\to E$ via
$\tau^L_g(\gamma)(x):=\gamma(gx)$
and $\tau^R_g(\gamma)(x):=\gamma(xg)$
for $x\in G$.
\end{defn}
\begin{la}\label{transla}
Let $r\in \N_0\cup\{\infty\}$
and~$E$ be a locally convex space.
If $r=0$, let $G$ be a locally compact group; otherwise,
let~$G$ be a Lie group. Let $g\in G$.
Then $\gamma\mto \tau^L_g(\gamma)$
defines isomorphisms
$C^r(G,E)\to C^r(G,E)$,
$C^r_K(G,E)\to C^r_{g^{-1}K}(G,E)$
$($for $K\sub G$ compact$)$
and $C^r_c(G,E)\to C^r_c(G,E)$
of topological vector spaces.
Likewise, $\gamma\mto\tau^R_g(\gamma)$
defines isomorphisms
$C^r(G,E)\to C^r(G,E)$,
$C^r_K(G,E)\to C^r_{Kg^{-1}}(G,E)$
$($for $K\sub G$ compact$)$
and $C^r_c(G,E)\to C^r_c(G,E)$
of topological vector spaces.
\end{la}
\begin{la}\label{normtrans}
For each $\ell\in \N_0$ such that $\ell\leq r$,
$\gamma\in C^r_c(G)$,
compact subset $K\sub G$ and $g\in G$,
we have
$\|\tau^L_g(\gamma)\|^L_{\ell,g^{-1}K}=\|\gamma\|^L_{\ell,K}$.
\end{la}
\begin{defn}
If $G$ is a locally compact group,
we let $\lambda_G$ be a Haar\linebreak
measure
on~$G$,
i.e., a left invariant, non-zero Radon measure (cf.\ \ref{measet}).
We let
$\Delta_G\colon G\to \;]0,\infty[$
be the modular function,
determined by $\lambda_G(Ex)=\Delta_G(x)\lambda_G(E)$
for all $x\in G$ and Borel sets $E\sub G$.
It is known that $\Delta_G$ is a continuous homomorphism \cite[2.24]{Fol}
(and hence smooth if~$G$ is a Lie group).
If $\gamma\colon G\to E$ is a mapping to
a vector space,
we define $\gamma^*\colon G\to E$ via
\[
\gamma^*(x):=\Delta_G(x^{-1})\gamma(x^{-1}).
\]
\end{defn}
It is clear from the definition that $(\gamma^*)^*=\gamma$.
\begin{la}\label{invoaut}
Let $E$ be a locally convex space and $r\in \N_0\cup\{\infty\}$.
If $r>0$, let $G$ be a Lie group; if $r=0$,
let $G$ be a locally compact group.
Then all of the following maps are isomorphisms
of topological vector spaces:
\[
\Theta\colon C^r(G,E)\to C^r(G,E)\,,\quad \gamma\mto \gamma^*\, ;
\]
\[
\Theta_K\colon C^r_K(G,E)\to C^r_{K^{-1}}(G,E)\,,\quad \gamma\mto \gamma^*\,,
\]
for $K$ a compact subset of~$G$; and
\[
\Theta_c\colon C^r_c(G,E)\to C^r_c(G,E)\,,\quad \gamma\mto \gamma^*\,.
\]
\end{la}
\noindent
{\bf Further facts concerning function spaces.}
Whenever we prove that mappings
to spaces of test functions are discontinuous,
the following embedding will allow us to reduce
to the case of scalar-valued test functions.
\begin{la}\label{Phiv}
For each $C^r$-manifold $M$ $($resp., locally compact space $M$, if $r=0)$,
locally convex space $E$
and $0\not=v\in E$, the map
\[
\Phi_v\colon C^r_c(M)\to C^r_c(M,E)\, ,\quad
\Phi_v(\gamma):=\gamma v
\]
is linear and a topological embedding
$($where $(\gamma v)(x):=\gamma(x)v)$.
\end{la}
The following related result will be used in Section~\ref{secsigma}.
\begin{la}\label{scavect}
Let $r\in \N_0\cup\{\infty\}$.
If $r>0$, let $M$ be a $C^r$-manifold;
if $r=0$, let $M$ be a Hausdorff topological space.
Then
the bilinear mapping\linebreak
$\Psi_E\colon C^r(M)\times E \to C^r(M,E)$, $(\gamma,v)\mto \gamma v$
is continuous.
If $M$ is locally compact and $K\sub M$ compact,
then also $\Psi_{K,E}\colon C^r_K(M)\times E \to C^r_K(M,E)$, $(\gamma,v)\mto \gamma v$
is continuous.
\end{la}
\begin{la}\label{towardsh}
Let $E$ be a locally convex space, and $r\in \N_0\cup\{\infty\}$.
If $r=0$, let $M$ be a locally compact space.
If $r>0$, let $M$ be a $C^r$-manifold. Then:
\begin{itemize}
\item[\rm(a)]
For each compact set $K\sub M$,
there exists a family $(\lambda_i)_{i\in I}$
of continuous linear maps $\lambda_i\colon E\to F_i$ to Banach spaces~$F_i$,
such that the topology on $C^r_K(M,E)$
is initial with respect to the linear mappings $C^r_K(M,\lambda_i):$\linebreak
$C^r_K(M,E)\to
C^r_K(M,F_i)$, $\gamma\mto \lambda_i\circ\gamma$
for $i\in I$.
\item[\rm(b)]
If $M$ is $\sigma$-compact, then there exists a family $(\lambda_i)_{i\in I}$
of continuous linear maps $\lambda_i\colon E\to F_i$ to Fr\'{e}chet spaces $F_i$,
such that the topology on $C^r_c(M,E)$
is initial with respect to the linear mappings $C^r_c(M,\lambda_i):$\linebreak
$C^r_c(M,E)\to C^r_c(M,F_i)$,
$\gamma\mto \lambda_i\circ \gamma$ for $i\in I$.
\item[\rm(c)]
If $M$ is paracompact and $B\sub C^r_c(M,E)$ is a bounded set, then $B\sub C^r_K(M,E)$
for some compact set
$K\sub M$.
\end{itemize}
\end{la}
\section{Basic facts concerning convolution}
Throughout this section, $G$ is a locally compact group, with left Haar measure $\lambda_G$,
and $b\colon E_1\times E_2\to F$
a continuous bilinear map between locally convex spaces.
As in the previous section, we refer to Appendix~\ref{appproofs} for all
proofs.
If~$G$ is not metrizable, we assume that~$F$ satisfies the
convex compactness property.
If $G$ is metrizable and not discrete, we assume that~$F$ is
sequentially complete or satisfies the metric convex compactness property.
Given $\gamma\in C(G,E_1)$ and $\eta\in C(G,E_2)$
such that $\gamma$ or $\eta$ has compact support,
we define
\[
\gamma*_b\eta\colon G\to F,\quad (\gamma*_b\eta)(x):=\int_Gb(\gamma(y),\eta(y^{-1}x))\,d\lambda_G(y)\, ,
\]
noting that the $E$-valued weak integral exists by Lemma~\ref{intex}
as the map $G\to F$, $y\mapsto b(\gamma(y),\eta(y^{-1}x))$
is continuous with support in the compact set
\begin{equation}\label{suppgand}
\Supp(\gamma)\cap x(\Supp(\eta))^{-1}\,.
\end{equation}
In particular,
\begin{equation}\label{goodpar}
(\gamma*_b\eta)(x)=
\int_{\Supp(\gamma)}b(\gamma(y),\eta(y^{-1}x))\,d\lambda_G(y)\,.
\end{equation}
If $b$ is understood, we simply write $\gamma*\eta:=\gamma*_b\eta$.
Consider the inversion map
$j_G\colon G\to G$, $g\mto g^{-1}$.
It is well known (see \cite[2.31]{Fol}) that
the image measure $j_G(\lambda_G)$
is of the form
\begin{equation}
j_G(\lambda_G)=\Delta_G(y^{-1})\,d\lambda_G(y)\, .
\end{equation}
Since $y^{-1}x=(x^{-1}y)^{-1}=j_G(L_{x^{-1}}(y))$, we infer
$j_G(L_{x^{-1}}(\lambda_G))=j_G(\lambda_G)=\Delta_G(y^{-1})\,d\lambda_G(y)$.
Now the Transformation Formula implies that\footnote{Apply continuous linear functionals
and use \cite[17.3 and 19.3]{Bau}.}
$(\gamma*_b\eta)(x)=
\int_Gb\circ (\gamma\circ L_x\circ j_G,\eta) \circ j_G\circ L_{x^{-1}}\, d\lambda_G
=\int_Gb\circ (\gamma\circ L_x\circ j_G,\eta) d((j_G\circ L_{x^{-1}})(\lambda_G))
=\int_Gb\circ (\gamma\circ L_x\circ j_G,\eta) \, d(\Delta_G(.^{-1})\odot \lambda_G)
=\int_Gb(\gamma(xz^{-1}),\eta(z)) \Delta_G(z^{-1})\, d\lambda_G(z)$.
Thus
\begin{equation}\label{altform}
(\gamma*_b\eta)(x)=\int_Gb(\gamma(xz^{-1}),\Delta_G(z^{-1})\eta(z))\, d\lambda_G(z)
\end{equation}
for all
$\gamma$, $\eta$ as above and $x\in G$. 
\begin{la}\label{convcts}
Let $K,L\sub G$ be closed sets and~$K$ or~$L$ be compact.
For all $\gamma\in C_K(G,E_1)$, $\eta\in C_L(G,E_2)$, we then have
$\gamma*_b\eta\in C_{KL}(G,F)$,
and
\begin{equation}\label{estsupp}
\Supp(\gamma*_b\eta)\sub\Supp(\gamma)\Supp(\eta)\sub KL\,.
\end{equation}
The bilinear map
$\beta\colon C_K(G,E_1)\times C_L(G,E_2)\to C_{KL}(G,F)$,
$(\gamma,\eta)\mapsto \gamma*_b\eta$
is continuous.
\end{la}
Let $G$ be a Lie group now.
\begin{prop}\label{diffbar}
Let $r,s\in \N_0\cup\{\infty\}$ and $K,L\sub G$ be closed subsets
such that $K$ or $L$ is compact.
Then $\gamma*_b\eta\in C_{KL}^{r+s}(G,F)$
for all
$\gamma\in C_K^r(G,E_1)$ and $\eta\in C_L^s(G,E_2)$,
with
\begin{equation}\label{hinein}
\cR_{w_j}\cdots \cR_{w_1}\cL_{v_i}\cdots\cL_{v_1}.(\gamma*_b\eta)
=
(\cR_{w_j}\cdots \cR_{w_1}.\gamma)*_b(\cL_{v_i}\cdots\cL_{v_1}.\eta)
\end{equation}
for all $i,j\in \N_0$ with $i\leq s$ and $j\leq r$,
and all $v_1,\ldots,v_i,w_1,\ldots, w_j\in T_1(G)$.
Moreover, the bilinear map
\[
\beta\colon C_K^r(G,E_1)\times C_L^s(G,E_2)\to C^{r+s}_{KL}(G,F)\,,\quad (\gamma,\eta)\mapsto \gamma*_b\eta
\]
is continuous.
\end{prop}
If $G$ is compact, then $C^r(G,E)=C^r_K(G,E)$ with $K:=G$,
for each $r\in \N_0\cup\{\infty\}$ and locally convex space $E$.
Hence Proposition~\ref{diffbar} yields as a special case:
\begin{cor}\label{cpcase}
If $G$ is compact, then the convolution map
\[
\beta\colon C^r(G,E_1)\times C^s(G,E_2)\to C^t(G,F),\quad
(\gamma,\eta)\mto \gamma *_b \eta
\]
is continuous, for all $r,s,t\in \N_0\cup\{\infty\}$
such that $t\leq r+s$.\Punkt
\end{cor}
To each continuous bilinear map
$b\colon E_1\times E_2\to F$ as before, we associate a continuous bilinear
map 
$b^\vee \colon E_2\times E_1\to F$ via
\begin{equation}\label{bvee}
b^\vee(x,y):=b(y,x)\quad\mbox{for $(x,y)\in E_2\times E_1$.}
\end{equation}
\begin{la}\label{invo}
If one of the maps $\gamma\in C^r(G,E_1)$
and $C^s(G,E_2)$ has compact support, then
$(\gamma*_b\eta)^*=\eta^* *_{b^\vee}\gamma^*$.
\end{la}
\begin{la}\label{efftra}
Assume that one of the maps $\gamma\in C^r(G,E_1)$
and $C^s(G,E_2)$ has compact support. Let $g\in G$. Then
\begin{itemize}
\item[\rm(a)]
$\tau^L_g(\gamma*_b\eta)=(\tau^L_g(\gamma))*_b\eta$
\item[\rm(b)]
$\tau^R_g(\gamma*_b\eta)=\gamma*_b(\tau^R_g(\eta))$.
\end{itemize}
\end{la}
\begin{la}\label{lasimpe}
Let $(G,r,s,b)$ be as in the introduction, $K\sub G$ be compact,
$\gamma\in C^r_K(G,E_1)$, $\eta\in C^s_c(G,E_2)$
and $q$, $p_1$, $p_2$ be continuous seminorms on $F$, $E_1$
and $E_2$, respectively, such that $q(b(x,y))\leq p_1(x)p_2(y)$
for all $(x,y)\in E_1\times E_2$. Let $k,\ell\in \N_0$
with $k\leq r$ and $\ell\leq s$. Then
\begin{eqnarray*}
\|\gamma*_b\eta\|^R_{k,q}\;\,& \leq &\|\gamma\|_{k,p_1}^R\|\eta\|_{p_2,\infty}\lambda_G(K)\\
\|\gamma*_b\eta\|^L_{\ell,q}\;\, & \leq & \|\gamma\|_{p_1,\infty}\|\eta\|_{\ell,p_2}^L\lambda_G(K)\\
\|\gamma*_b\eta\|^{R,L}_{k,\ell,q} & \leq & \|\gamma\|_{k,p_1}^R\|\eta\|_{\ell,p_2}^L\lambda_G(K).
\end{eqnarray*}
\end{la}
For a definition and necessary background on
hypocontinuous bilinear maps, the reader is referred to Appendix~\ref{hypoapp}.
In Appendix~\ref{appproofs}, we also prove:
\begin{prop}\label{hypo}\label{bihyp}
For all $(G,r,s,t,b)$ as in the introduction,
the convolution map $\beta\colon C^r_c(G,E_1)\times C^s_c(G,E_2)\to
C^t_c(G,F)$, $(\gamma,\eta)\mto \gamma *_b\eta$ is hypocontinuous.
\end{prop}
\section{Facts on measures and their convolution}\label{secmeasset}
In this section, we fix our measure-theoretic setting
and state basic definitions and facts concerning spaces
of complex Radon measure and convolution of
complex Radon measures. As before, proofs
can be looked up in Appendix~\ref{appproofs}.
\begin{numba}\label{measet}
{\bf The setting.}
If~$X$ is a Hausdorff topological space,
we write~$\cB(X)$ for the $\sigma$-algebra of Borel sets
(which is generated by the set of open subsets of~$X$).
A positive measure $\mu\colon \cB(X)\to [0,\infty]$
is called a \emph{Borel measure} if $\mu(K)<\infty$
for each compact subset $K\sub X$.
Following~\cite{BCR}, we shall call a Borel measure $\mu$ on~$X$
a \emph{Radon measure} if~$\mu$ is inner regular in the sense that $\mu(A)=\sup\{ \mu(K)\colon
K\sub A\, \text{compact}\}$ for each $A\in \cB(X)$.
The support $\Supp(\mu)$ of a Radon measure is the
smallest closed subset
of~$X$ such that\linebreak
$\mu(X\setminus\Supp(\mu))=0$.
A complex measure $\mu\colon \cB(X)\to\C$ is called a \emph{complex Radon measure}
on~$X$ if the
associated total variation measure $|\mu|$ (as in \cite[6.2]{Rud})
is a (finite) positive Radon measure.
In this case, we set $\Supp(\mu):=\Supp(|\mu|)$.
The total
variation norm of $\mu$
is defined via $\|\mu\|:=|\mu|(X)$.
We let $\Mea(X)$ be the space of all complex Radon measures on~$X$.
Given a compact set $K\sub X$,
we let $\Mea_K(X)$ be the space of
all $\mu\in \Mea(X)$ such that $\Supp(\mu)\sub K$.
It is clear that the restriction map $(\Mea_K(X),\|.\|)\to (\Mea(K),\|.\|)$
is an
isometric
isomorphism, and hence $\Mea_K(X)\isom \Mea(K)\isom (C(K)',\|.\|_{\op })$
is a
Banach space (using the Riesz Representation Theorem, \cite[6.19]{Rud}).
We give $\Mea_c(X):=\bigcup_K \Mea_K(X)$ the locally convex direct
limit topology, and note that it is Hausdorff (being finer than the normable topology arising from
the total variation norm). We let $\Mea(X)_+$ be the set of finite positive Radon measures
on~$X$, $\Mea_K(X)_+$ be the subset of Radon measures
supported in a given compact set $K\sub X$, and $\Mea_c(X)_+:=\bigcup_K\, \Mea_K(X)_+$.
If $G$ is a Hausdorff topological group, with group multiplication
$m\colon G\times G\to G$, we let $\mu\otimes \nu$ be the Radon product measure
of $\mu,\nu\in \Mea(G)_+$ (see \cite[2.1.11]{BCR}).
We define $\mu\otimes \nu$ for $\mu,\nu\in \Mea(G)$
via bilinear extension; then $|\mu\otimes\nu|\leq |\mu|\otimes|\nu|$
(see, e.g.,  \cite[(5.4)]{GLP}).
The convolution of $\mu,\nu\in \Mea(G)$ is defined as the measure
$\mu*\nu:=m_*(\mu\otimes\nu)$
taking $A\in \cB(G)$ to $(\mu\otimes \nu)(m^{-1}(A))$
(cf.\ \cite[2.1.16]{BCR}).
Since $|\mu*\nu|=|m_*(\mu\otimes\nu)|\leq m_*(|\mu\otimes \nu|)
\leq m_*(|\mu|\otimes|\nu|)=
|\mu|*|\nu|$, one deduces that
$\|\mu*\nu\|\leq (|\mu|*|\nu|)(G)=(|\mu|\otimes |\nu|)(G\times G)=|\mu|(G)|\nu|(G)=\|\mu\|\,\|\nu\|$.
We shall use that\footnote{If $z\in G\setminus\Supp(\mu)\Supp(\nu)=:U$ and $x,y\in G$ with $xy=z$,
then $x\not\in \Supp(\mu)$ or $y\not\in \Supp(\nu)$.
Hence $m^{-1}(U)\cap(\Supp(\mu)\times\Supp(\nu))=\emptyset$
and $|\mu*\nu|(U)\leq (|\mu|*|\nu|)(U)=(|\mu|\otimes |\nu|)(m^{-1}(U))
\leq (|\mu|\otimes|\nu|)(G\times (G\setminus \Supp(\nu)))+(|\mu|\otimes |\nu|)((G\setminus\Supp(\mu))\times G)=0$.}
\begin{equation}\label{asexpected}
\Supp(\mu*\nu)\sub \Supp(\mu)\Supp(\nu)\qquad\mbox{for all $\,\mu,\nu\in \Mea_c(G)$.}\vspace{2mm}
\end{equation}
\end{numba}
\begin{la}\label{summeas}
Let $X$ be a Hausdorff topological space and $(A_j)_{j\in J}$ be a family of Borel
subsets of~$X$, such that $J_K:=\{j\in J\colon A_j\cap K\not=\emptyset\}$
is finite for each compact subset $K\sub X$.
Then the map
\[
\Phi\colon \Mea_c(X)\to\bigoplus_{j\in J }\Mea(A_j)\,, \quad \mu\mto (\mu|_{\cB(A_j)})_{j\in J}
\]
is continuous and linear.
\end{la}
Recall the notation $\rho\odot\mu$ for $\rho\,d\mu$.
\begin{la}\label{m-emb}
Let $X$ be a Hausdorff topological space
which is hemicompact,
and $K_1\sub K_2\sub\cdots$ be compact subsets of~$X$ such that
each compact subset of~$X$ is contained in some $K_n$.
Let $K_0:=\emptyset$.
Given $\mu\in \Mea_c(X)$, we have $\mu_n:=\one_{K_n\setminus K_{n-1}}\odot \mu\in \Mea_{K_n}(X)$
for each $n\in\N$, and the map
\[
\Phi\colon \Mea_c(X)\to\bigoplus_{n\in\N}\Mea_{K_n}(X)\,, \quad \mu\mto (\mu_n)_{n\in\N}
\]
is linear and a topological embedding.
\end{la}
If $X$ is a locally compact space, $\mu\geq 0$ a Radon measure on~$X$
and $K\sub X$ a compact set, we define $(L^1(X,\mu),\|.\|_{L^1})$
as usual and let $L^1_K(X,\mu)$
be the set of all $[\gamma]\in L^1(X,\mu)$
vanishing $\mu$-almost everywhere outside~$K$.
We equip $L^1_K(X,\mu)$ with the topology induced by $L^1(X,\mu)$,
and $L^1_c(X,\mu):=\bigcup_K\, L^1_K(X,\mu)$
with the locally convex direct limit topology.
We abbreviate
$L^1_c(G):=L^1_c(G,\lambda_G)$.
\begin{la}\label{intomas}
For each locally compact space $X$ and Radon measure
$\mu\geq 0$ on $X$, the map
\[
\Phi\colon C_c(X)\to \Mea_c(X),\quad \gamma\mto \gamma\odot \mu
\]
is continuous and linear, and so is
$\Psi\colon L^1_c(X,\mu) \to \Mea_c(X)$, $\gamma\mto \gamma\odot \mu$.
\end{la}
As is well-known, the definitions of convolution of functions and that of
measures are compatible with one another (cf.\ \cite[p.\,50]{Fol}):
\begin{la}\label{fuvsm}
If $G$ is a locally compact group, with left Haar measure $\lambda_G$,
then $(\gamma\odot \lambda_G)*(\eta\odot \lambda_G)=(\gamma*\eta)\odot\lambda_G$
for all $\gamma,\eta\in C_c(G)$
$($and, more generally, for $\gamma,\eta\in L^1_c(G))$.\,\Punkt
\end{la}
\section{Convolution on non-{\boldmath$\sigma$}-compact groups}\label{aness}
We prove that convolution of
test functions on a non-$\sigma$-compact group
is always discontinuous (Proposition~\ref{disco}).
Notably, this shows the necessity of condition~(a)
in Theorem~A.
It is efficient to discuss convolution
of measures in parallel,
and all our results concerning it.
\cite[Lemma]{Bis} is essential:
\begin{la}\label{bisg}
Let $I$ be an uncountable set.
Then there exists a function\linebreak
$g\colon I\times I\to \; ]0,\infty[$
such that for each $v\colon I \to \; ]0,\infty[$,
there exist $i,j\in I$ such that $v(i)v(j)>g(i,j)$.\Punkt
\end{la}
\begin{la}\label{technmeas}
Let $G$ be a Hausdorff topological group and
$W\sub \Mea_c(G)$ be a cone $($i.e., $[0,\infty[\, W\sub W)$,
equipped with a topology $\cO$ making the map $m_\mu\colon [0,\infty[\to W$, $r\mto r\mu$ continuous at~$0$
for each $\mu\in W$.
Assume that there exists an uncountable set~$I$ and families $(Y_i)_{i\in I}$
and $(Z_i)_{i\in I}$ of Borel sets in~$G$ such that
\[
I_K:=\{(i,j)\in I\times I\colon Y_iZ_j\cap K\not=\emptyset\}
\]
is finite for each compact subset $K\sub G$,
and there exist non-zero measures $\mu_i,\nu_i\in \Mea_c(G)_+\cap W$
such that $\Supp(\mu_i)\sub Y_i$ and $\Supp(\nu_i)\sub Z_i$,
for all $i\in I$.
Then the convolution map
\[
\beta\colon (W,\cO)\times (W,\cO) \to \Mea_c(G),\quad (\mu,\nu)\mto \mu *\nu
\]
is discontinuous $($with respect to the usual locally convex direct limit topology on the right hand side$)$.
\end{la}
\begin{proof}
After passing to a positive multiple, we may assume that
$\|\mu_i\|=\|\nu_i\|=1$ for all $i\in I$.
Let $g\colon I\times I\to \;]0,\infty[$ be as in Lemma~\ref{bisg}.
By Lemma~\ref{summeas}, the restriction maps combine to a continuous linear mapping\linebreak
$\Mea_c(G)\to \bigoplus_{(i,j)\in I\times I}\Mea(Y_iZ_j)$.
Hence the set
\[
S:=\{\mu \in \Mea_c(G) \colon (\forall i,j\in I)\; |\mu|(Y_iZ_j)<g(i,j)\}
\]
is an open $0$-neighbourhood in $\Mea_c(G)$.
We now show that
\begin{equation}\label{leave}
\beta(U\times V)=U*V\not\sub S
\end{equation}
for any $0$-neighbourhoods $U\sub W$
and $V\sub W$.
Hence $\beta$ will be discontinuous at $(0,0)$.
Since~$U$ is a $0$-neighbourhood and $m_{\mu_i}$ is continuous
at~$0$ for $i\in I$, we find $\ve_i>0$
such that $\ve_i \mu_i\in U$.
Likewise, we find $\theta_i>0$ such that $\theta_i\nu_i\in V$.
By choice of~$g$, there exist $i,j\in I$ such that $\ve_i\theta_j>g(i,j)$.
Since $(\ve_i \mu_i)*(\theta_j \nu_j)\in \Mea_c(G)_+$
and
$\Supp((\ve_i \mu_i)*(\theta_j \nu_j))\sub \Supp(\mu_i)\Supp(\nu_j)\sub Y_iZ_j$, we obtain
\begin{eqnarray*}
|(\ve_i \mu_i)*(\theta_j \nu_j)|(Y_iZ_j) & =&
(\ve_i \mu_i*\theta_j \nu_j)(Y_iZ_j)\; =\; 
(\ve_i \mu_i*\theta_j \nu_j)(G)\\
&=&
\ve_i\theta_j \,\mu_i * \nu_j
\; =\; 
\ve_i\theta_i \, (\mu_i \otimes  \nu_j)(G\times G)\\
& =&
\ve_i\theta_j \, \mu_i(G)\nu_j(G)
\;=\; \ve_i\theta_j>g(i,j).
\end{eqnarray*}
Hence $(\ve_i\mu_i)*(\theta_j\nu_j)\not\in S$, establishing (\ref{leave}).
\end{proof}
{\bf Proof of Theorem~B.}
(a) If $G$ is hemicompact, define $\Phi$ for $X:=G$ as in Lemma~\ref{m-emb}.
For $i,j\in \N$, let $f_{i,j}\colon \Mea_{K_i}(G)\times \Mea_{K_j}(G)\to \Mea_{K_iK_j}(G)\sub \Mea_c(G)$
be the convolution map. Then $f_{i,j}$ is continuous,
because $\|f_{i,j}(\mu,\nu)\|=\|\mu*\nu\|\leq \|\mu\|\,\|\nu\|$.
Abbreviate $S:=\bigoplus_{i\in \N}\Mea_{K_i}(X)$.
Since each of the spaces $\Mea_{K_i}(X)$ is normable, it follows
that the bilinear map
$f\colon S\times S\to \Mea_c(G)$, $f((\gamma_i)_{i\in \N},(\eta_j)_{j\in \N}):=\sum_{i,j\in \N} f_{i,j}(\gamma_i,\eta_j)$
is continuous \cite[Corollary~2.4]{Glo}.
Hence also the convolution map~$\beta$ on $\Mea_c(G)$ is continuous,
as it can be written in the form $\beta=f\circ (\Phi\times \Phi)$.

(b) If $G$ is spacious, then there exist uncountable subsets $A,B\sub G$ such that
$\{(a,b)\in A\times B\colon ab\in K\}$
is finite for each compact set $K\sub G$.\linebreak
After replacing~$A$ and~$B$ by subsets whose cardinality is the smallest\linebreak
uncountable
cardinal~$\aleph_1$, we may assume that there exists a bijection\linebreak
$f\colon A\to B$.
Then $(\{a\})_{a\in A}$
and $(\{f(a)\})_{a\in A}$
are families of (singleton) subsets of~$G$ such that
$\{(a,a')\in A\times A\colon \{a\} \{f(a')\}\cap K\not=\emptyset \}$
is finite for each
compact subset $K\sub G$.
Now define $W:=\Mea_c(G)$, with its usual topology,
and note that the point measures $\mu_a:=\delta_a$
at~$a$ and $\nu_{a'}:=\delta_{f(a')}$
at~$f(a')$ on~$G$
are contained in $W\cap \Mea_{\{a\}}(G)_+$
and $W\cap \Mea_{\{f(a')\}}(G)_+$,
respectively.
Thus Lemma~\ref{technmeas}
shows that $\beta$ is not continuous.\vspace{1.5mm}\Punkt

\begin{prop}\label{disco}
Let $(G,r,s,t, b)$ and $\beta\colon C^r_c(G,E_1)\times C^s_c(G,E_2)\to C^t_c(G,F)$
be as in the introduction.
If $G$ is not $\sigma$-compact,
then $\beta$ is not continuous.
\end{prop}
Our proof of Proposition~\ref{disco}
uses a property of non-$\sigma$-compact groups:
\begin{la}\label{twosetnew}
Let $G$ be a locally compact group which is not $\sigma$-compact,
and $U\sub G$ be a $\sigma$-compact, open subgroup of $G$.
Then
there exist disjoint uncountable subsets $A,B\sub G$
such that $A$ and $B$ have the same cardinality,
\begin{equation}\label{mutdis}
(\forall (a,b),(a',b')\in A\times B) \quad aUb\cap a'Ub'\not=\emptyset\impl (a,b)=(a',b')\,,
\end{equation}
and $(aUb)_{(a,b)\in A\times B}$ is locally finite.\,\Punkt
\end{la}
\begin{proof}
Let $\omega_1$ be the first uncountable ordinal. Fix a well-ordering $\preceq$
on $G$.
We prove the following assertion $P(\theta)$ for all ordinals $\theta\leq \omega_1$,
by transfinite induction:

$P(\theta)$: There exist uniquely determined families $(A_\alpha)_{\alpha\leq\theta}$ and $(B_\alpha)_{\alpha\leq\theta}$
of subsets $A_\alpha, B_\alpha\sub G$
and a unique family $(f_\alpha)_{\alpha\leq \theta}$ of bijections $f_\alpha\colon A_\alpha\to B_\alpha$
such that $A_0=B_0=\emptyset$, $A_\alpha$ and $B_\alpha$ are countable
for all $\alpha\leq\theta$ such that $\alpha\not=\omega_1$, moreover
\begin{eqnarray}
\lefteqn{(\forall \alpha\leq \theta)\quad (\alpha+1\leq\theta)\impl}\qquad \notag \\
& & \hspace*{-15mm}\left\{
\begin{array}{cclll}
A_{\alpha+1} & = & A_\alpha\cup\{x_\alpha\} &
\mbox{with} & x_\alpha=\min (G\setminus \langle \, U\, \cup \, A_\alpha\, \cup \, B_\alpha\, \rangle) \mbox{ and}\\
B_{\alpha+1} & = & B_\alpha\cup\{y_\alpha\} &\mbox{with} & y_\alpha=\min(G\setminus \langle U\cup A_{\alpha+1}\cup B_\alpha\rangle),
\end{array}
\right.
\label{propersub}
\end{eqnarray}
$A_\alpha=\bigcup_{\beta<\alpha}A_\beta$ and $B_\alpha=\bigcup_{\beta<\alpha}B_\beta$
for all limit ordinals $\alpha\leq\theta$,
and $f_\alpha|_{A_\beta}=f_\beta$ for all $\beta\leq \alpha\leq\theta$.

Here, the minima in (\ref{propersub}) refer to the chosen well-ordering on~$G$.
Because $A_\alpha$ and $B_\alpha$ are countable if
$\alpha+1\leq \theta$, the group
$\langle \, U\, \cup \, A_\alpha\, \cup \, B_\alpha\, \rangle$
is $\sigma$-compact and hence not all of~$G$.
As a consequence, its complement
$G\setminus \langle \, U\, \cup \, A_\alpha\, \cup \, B_\alpha\, \rangle$ is non-empty and
the first minimum in~(\ref{propersub})
makes sense. Similarly, the second minimum makes sense.

To prove $P(\theta)$ by transfinite induction,
note that $P(0)$ is satisfied if and only if $A_0=B_0=f_0=\emptyset$.

If $\theta$ is a non-zero limit ordinal and $P(\theta')$ holds for all $\theta'<\theta$,
write $((A^{\theta'}_\alpha)_{\alpha\leq \theta'}, (B^{\theta'}_\alpha)_{\alpha\leq\theta},(f_\alpha^{\theta'})_{\alpha\leq\theta'})$
for the triple $((A_\alpha)_{\alpha\leq \theta'}, (B_\alpha)_{\alpha\leq\theta'},(f_\alpha)_{\alpha\leq\theta'})$
that is uniquely determined by $P(\theta')$.
If $\theta''\leq \theta'<\theta$,
the uniqueness in $P(\theta'')$ implies that $A_\alpha^{\theta''}=A_\alpha^{\theta'}$,
$B_\alpha^{\theta''}=B_\alpha^{\theta'}$ and $f_\alpha^{\theta''}=f_\alpha^{\theta'}$
for all $\alpha\leq \theta''$.
For $\alpha<\theta$, choose $\theta'<\theta$ such that $\alpha\leq \theta'$;
then $A_\alpha:=A^{\theta'}_\alpha$,
$B_\alpha:=B^{\theta'}_\alpha$,
and $f_\alpha:=f^{\theta'}_\alpha$
are independent of the choice of~$\theta'$
(as just observed). We also set $A_\theta:=\bigcup_{\alpha<\theta}A_\alpha$,
$B_\theta:=\bigcup_{\alpha<\theta}B_\alpha$
and $f_\theta:=\bigcup_{\alpha<\theta}f_\alpha$.
Then $P(\theta)$ holds.

If $\theta=\theta'+1$,
let
$A_\alpha:=A^{\theta'}_\alpha$,
$B_\alpha:=B^{\theta'}_\alpha$,
and $f_\alpha:=f^{\theta'}_\alpha$
for $\alpha\leq\theta'$.
Define $A_\theta :=  A_{\theta}\cup\{x_{\theta'}\}$
with $x_{\theta'}=\min (G\setminus \langle \, U\, \cup \, A_{\theta'}\, \cup \, B_{\theta'}\, \rangle)$.
Also, define
$B_\theta:= B_{\theta'}\cup\{y_{\theta'}\}$
with $y_{\theta'}:=\min(G\setminus \langle U\cup A_\theta\cup B_{\theta'}\rangle)$.
Then $P(\theta)$ is satisfied. The inductive proof is complete.

Now, set $A:=A_{\omega_1}$ and
$B:=B_{\omega_1}$.
These are uncountable sets,
as they can be considered as the disjoint unions
$A=\bigcup_{\alpha<\omega_1}\{x_\alpha\}$ and
$B=\bigcup_{\alpha<\omega_1}\{y_\alpha\}$.
Moreover, $f_{\omega_1}\colon A\to B$ is a bijection,
and (\ref{mutdis}) can be inferred from $P(\omega_1)$.
In fact, assume that
\[
x_\alpha U y_\beta \cap x_\gamma U y_\delta \not=\emptyset\,.
\]
Thus, there exist $u,w\in U$ such that $x_\alpha u y_\beta  = x_\gamma w y_\delta $.
Let $\theta:=\max\{\alpha,\beta,\gamma,\delta\}$.
If $\beta=\theta$ and $\delta<\theta$, then 
$H:=\langle U\cup A_{\theta+1}\cup B_\theta\rangle$
would be a subgroup containing $U$ and all of $x_\alpha, x_\gamma$ and $y_\delta$.
Hence $y_\theta= y_\beta= u^{-1}x_\alpha^{-1}x_\gamma w y_\delta \in H$,
contradicting (\ref{propersub}).
Hence $\beta=\theta$ implies $\delta=\theta=\beta$.
Thus $x_\alpha u = x_\gamma w$ in this case. If $\alpha>\gamma$,
let $I:=\langle U\cup A_\alpha\rangle$.
Then $u,w,x_\gamma\in I$ and hence also $x_\alpha=x_\gamma w u^{-1}\in I$,
contradicting (\ref{propersub}).
The same argument excludes the case $\alpha<\gamma$,
and thus $\alpha=\gamma$.

Likewise, $\delta=\theta$ implies $\beta=\delta$,
from which $\alpha=\gamma$ follows as just shown.

If $\beta<\theta$ and $\delta<\theta$,
we may assume that $\alpha=\theta$
(the case $\gamma=\theta$ is analogous).
If we would have $\gamma<\alpha$,
then $H:=\langle U\cup A_\alpha\cup B_\alpha \rangle$ would be a subgroup
containing $\{u,y_\beta,x_\gamma,w, y_\delta\}$.
Hence $x_\alpha =x_\gamma w y_\delta y_\beta^{-1} u^{-1}\in H$,
contradicting~(\ref{propersub}).
Thus $\alpha=\gamma$. But then
$u y_\beta  = w y_\delta $.
Without loss of generality $\beta\leq\delta$.
If we would have $\beta<\delta$,
then $I:=\langle U\cup B_\delta\rangle$
would be a subgroup containing $\{u,y_\beta,w\}$.
Hence $y_\delta = w^{-1}u y_\beta$ would be in~$I$,
contradicting~(\ref{propersub}). Thus~(\ref{mutdis})
holds.

If $K\sub G$ is a compact set,
let $\Phi$ be the set
of all pairs $(\alpha,\beta)$ with $\alpha,\beta<\omega_1$
such that $x_\alpha Uy_\beta \cap K\not=\emptyset$.
To see that $\Phi$ is finite,
let us suppose that~$\Phi$ was infinite and derive a contradiction.
Case~1: Assume that $\Theta:=\{\max\{\alpha,\beta\}\colon (\alpha,\beta)\in \Phi\}$
is finite. Then (1a) the set $C:=\{\beta\leq \alpha_0\colon (\alpha_0,\beta)\in \Phi\}$
is infinite for some $\alpha_0<\omega_1$, or (1b)
the set $D:=\{\alpha\leq \beta_0\colon (\alpha,\beta_0)\in \Phi\}$
is infinite for some $\beta_0< \omega_1$.
In case (1a), $K$ meets $x_{\alpha_0}U y_\beta$ for all $\beta\in C$
(which are disjoint sets),
and hence the compact set $x_{\alpha_0}^{-1} K$ meets $U y_\beta$
for all $\beta\in C$, and also these sets are disjoint.
But the set $U \backslash G$ of all right cosets of $U$
is an open cover of $x_{\alpha_0}^{-1}K$ by disjoint open sets,
and hence $\{S\in U\backslash G\colon x_{\alpha_0}^{-1}K\cap S\not=\emptyset\}$
must be finite, contradiction.
In case (1b), $K$ meets $x_\alpha U y_{\beta_0}$ for all $\alpha\in D$
(which are disjoint sets),
and hence the compact set $Ky_{\beta_0}^{-1}$ meets $x_\alpha U $
for all $\alpha \in D$, and also these sets are disjoint.
But the set $G/U$ of all left cosets of $U$
is an open cover of $K y_{\beta_0}^{-1}$ by disjoint open sets,
and hence $\{S\in G/U\colon Ky_{\beta_0}^{-1}\cap S\not=\emptyset\}$
must be finite, contradiction.
Case~2: Assume that $\Theta$ is infinite.
For each $\theta\in \Theta$,
pick $(\alpha_\theta,\beta_\theta)\in \Phi$ such that $\max\{\alpha_\theta,\beta_\theta\}=\theta$.
Also, pick $z_\theta\in K\cap x_{\alpha_\theta}U y_{\beta_\theta}$.
Then (2a) $C:=\{\theta\in \Theta\colon \theta=\beta_\theta\}$
is infinite or (2b) the set
$D:=\{\theta\in \Theta\colon \theta=\alpha_\theta>\beta_\theta\}$ is infinite.
In case (2a),
if $\theta,\theta'\in C$ and $\theta<\theta'$,
then $x_{\alpha_\theta} U y_{\beta_\theta}\sub \langle U\cup A_{\theta'+1}\cup B_{\theta'}\rangle=:H$
and $x_{\alpha_{\theta'}}\in H$ (as $\alpha_{\theta'}\leq\theta'$).
Since $y_{\beta_{\theta'}}=y_{\theta'}\not\in H$ by (\ref{propersub}),
we have $H\cap H y_{\beta_{\theta'}}=\emptyset$
and hence $x_{\alpha_\theta} U y_{\beta_\theta} U\cap x_{\alpha_{\theta'}}Uy_{\beta_{\theta'}}=\emptyset$,
entailing that $z_\theta$ and $z_{\theta'}$ lie in different
left cosets of~$U$, i.e., $z_\theta U\cap z_{\theta'}U=\emptyset$.
Hence $K$ meets infinitely many left cosets of~$U$,
contradiction.
In case (2b),
if $\theta,\theta'\in D$ and $\theta<\theta'$,
then $x_{\alpha_\theta} U y_{\beta_\theta}\sub \langle U\cup A_{\theta'} \cup B_{\theta'}\rangle=:H$
and $y_{\beta_{\theta'}}\in H$ (as $\beta_{\theta'}<\theta'$).
Since $x_{\alpha_{\theta'}}=x_{\theta'}\not\in H$ by (\ref{propersub}),
we have $H\cap x_{\alpha_{\theta'}}H=\emptyset$
and hence $x_{\alpha_\theta} U y_{\beta_\theta} U\cap x_{\alpha_{\theta'}}Uy_{\beta_{\theta'}}=\emptyset$,
entailing that $z_\theta$ and $z_{\theta'}$ lie in different
left cosets of~$U$.
Hence $K$ meets infinitely many left cosets of~$U$,
contradiction.
\end{proof}
{\bf Proof of Proposition~\ref{disco}.} We write $\beta_b$ in place of~$\beta$.

As $b\not=0$, there exist non-zero vectors
$v\in E_1$, $w\in E_2$ and $z\in F$ such that $\beta(v,w)=z$.
Let
$\Phi_v\colon C^r_c(G)\to C^r_c(G,E_1)$,
$\Phi_w\colon C^s_c(G)\to C^s_c(G,E_2)$
and $\Phi_z\colon \!C^t_c(G)\!\to\! C^t_c(G,F)$
be the linear
topological embeddings from Lemma~\ref{Phiv}.
If $c\colon \R\times \R\to\R$, $(s,t)\mto s\cdot t$
is the scalar multiplication, then
\begin{equation}\label{redsca}
\beta_b\circ (\Phi_v\times \Phi_w)=\Phi_z\circ \beta_c\,.
\end{equation}
Hence $\beta_b$ will be discontinuous
if we can show that $\beta_c$ is discontinuous.
Let $\theta\colon \Mea_c(G)\times \Mea_c(G)\to \Mea_c(G)$
be convolution of measures.
Let $U\sub G$ be a $\sigma$-compact open subgroup.
As we assume that $G$ is not $\sigma$-compact,
Lemma~\ref{twosetnew}
provides uncountable subsets $A,B\sub G$
and a bijection $f\colon A\to B$, such that $(aUb)_{(a,b)\in A\times B}$
is a locally finite family of disjoint open subsets of~$G$.
Define $Y_a:=aU$ and $Z_a:=Uf(a)$ for $a\in A$.
Then $(Y_aZ_b)_{(a,b)\in A\times A}$
is a locally finite family of disjoint open
subsets of~$G$.
The map~$\Phi$ from Lemma~\ref{intomas}
(applied with $\mu:=\lambda_G$)
is continuous linear and injective.
We endow its image $W:=\im(\Phi)\sub \Mea_c(G)$
with the topology making $\Phi$ a homeomorphism onto~$W$.
For all $a\in A$,
there exist non-zero functions $g_a \in C_c^r(G)$ and $h_a\in C^s_c(G)$ with
$g_a,h_a\geq 0$ pointwise and $\Supp(g_a)\sub Y_a$,
$\Supp(h_a)\sub Z_a$.
Now the hypotheses of Lemma~\ref{technmeas} are satisfied with $\mu_a:=g_a\odot \lambda_G$
and $\nu_a:=h_a\odot \lambda_G$.
Hence $\theta|_{W\times W}$
is discontinuous. But $\beta_c=\theta\circ (\Phi\times \Phi)$
(see Lemma~\ref{fuvsm}), entailing that also $\beta_c$ is discontinuous.\,\Punkt

\begin{rem}\label{remspa}
A locally compact group~$G$ is spacious if and only if it is not $\sigma$-compact.
In fact, if $G$ is spacious, then the convolution map\linebreak
$\beta\colon \Mea_c(G)\times \Mea_c(G)\to \Mea_c(G)$
is discontinuous (see Theorem~B\,(b)),
whence~$G$ is not hemicompact (by Theorem~B\,(a))
and hence not $\sigma$-compact.
If $G$ is not $\sigma$-compact, then~$G$ is spacious,
as a consequence of Lemma~\ref{twosetnew}.
\end{rem}
\begin{cor}\label{L1c}
For $G$ a locally compact group,
the convolution mapping\linebreak
$\beta\colon L^1_c(G)\times L^1_c(G)\to L^1_c(G)$
is continuous if and only if $G$ is $\sigma$-compact.
\end{cor}
\begin{proof}
If $G$ is $\sigma$-compact, using the local compactness we find
compact subsets $K_n\sub G$ such that $G=\bigcup_{n=1}^\infty K_n$
and $K_n$ is contained in the interior of~$K_{n+1}$. Set $K_0:=\emptyset$
and abbreviate $S:=\bigoplus_{n\in \N} L^1_{K_n}(G)$.
Then the map
\[
\Phi\colon L^1_c(G)\to S\,,\;\; 
\Phi(\gamma):=(\one_{K_i\setminus K_{i-1}}\gamma)_{i\in \N}
\]
is linear, injective, and continuous (as in the proof of Lemma~\ref{m-emb},
using that $\|\one_{K_i\setminus K_{i-1}}\gamma\|_{L^1}\leq
\|\gamma\|_{L^1}$).
Since $\|\gamma*\eta\|_{L^1}\leq \|\gamma\|_{L^1}\|\eta\|_{L^1}$,
the restriction $f_{i,j}$ of $\beta$ to $L^1_{K_i}(G)\times L^1_{K_j}(G)$ is continuous for
all $i,j\in \N$. As all of the spaces $L^1_{K_i}(G)$ are normable,
\cite[Corollary~2.4]{Glo} shows that $f\colon S\times S\to L^1_c(G)$,
$f((\gamma_i)_{i\in\N},(\eta_j)_{j\in \N}):=\sum_{i,j\in \N}f_{i,j}(\gamma_i,\eta_j)$
is continuous.
Hence $\beta=f\circ (\Phi\times\Phi)$
is continuous as well.

If $G$ is not $\sigma$-compact,
let $\theta\colon \Mea_c(G)\times \Mea_c(G)\to \Mea_c(G)$
be convolution of measures.
Let $U\sub G$ be a $\sigma$-compact open subgroup,
$(Y_a)_{a\in A}$, $(Z_a)_{a\in A}$,
$g_a$ and $h_a\in C_c(G)$ be as in the proof of Proposition~\ref{disco}.
The map~$\Psi$ from Lemma~\ref{intomas}
(applied with $\mu:=\lambda_G$)
is continuous linear and injective.
We endow its image $W:=\im(\Psi)\sub \Mea_c(G)$
with the topology making~$\Psi$ a homeo\-morphism onto~$W$.
Now the hypotheses of Lemma~\ref{technmeas} are satisfied,
whence $\theta|_{W\times W}$
is discontinuous and hence also $\beta=\theta\circ (\Psi\times \Psi)$.
\end{proof}
\section{Convolution of {\boldmath$C^k_c$}-maps and {\boldmath$C^\infty_c$}-maps}\label{bness}
In this section, we prove the necessity of condition~(b) in Theorem~A.
Thus, we assume that~(b) is violated
and deduce that $\beta_b:=\beta$ is discontinuous.
In view of Proposition~\ref{disco},
it suffices to show this if~(a) is satisfied.

Thus, let $G$ be a $\sigma$-compact, non-discrete, non-compact
Lie group, and $\cg:=T_1(G)$.
If $r=\infty$ and $s<\infty$, we have $\beta_b(\gamma,\eta)=(\beta_{b^\vee}(\eta^*,\gamma^*))^*$
for $(\gamma,\eta)\in C^\infty_c(G,E_1)\times C^s_c(G,E_2)$,
where $\beta_{b^\vee}\colon C^s_c(G,E_2)\times C^\infty_c(G,E_1)\to C^\infty_c(G,F)$
and $*$ stands for the involutions
on $C^s_c(G,E_2)$, $C^\infty_c(G,E_1)$ and $C^\infty_c(G,F)$,
respectively, which are isomorphisms of topological vector spaces by Lemma~\ref{invoaut}.
Hence discontinuity of $\beta_{b^\vee}$
will entail discontinuity of~$\beta_b$.
It therefore suffices to assume that $r<\infty$ and $s=\infty$
in the rest of the proof.

I show that the convolution map
$\beta\colon C^r_c(G,E_1)\times C^\infty_c(G,E_2)\to C^\infty_c(G,F)$
is discontinuous. As in the proof of Proposition~\ref{disco},
we may assume that $E_1=E_2=F=\R$
and that $b\colon \R\times \R\to\R$
is multiplication, for the proof of discontinuity.
Let $K\sub G$ be a compact identity neighbourhood, and
$M \sub G$ be a relatively compact, open set such that $KK\sub M$.
There exists a sequence $(x_i)_{i\in \N}$
in $G$ such that $(x_iM)_{i\in \N}$ is locally
finite. For each $i\in \N$, let $h_i\in C^\infty_c(G)$
be a function such that $\Supp(h_i)\sub x_iM$ and $h_i=1$
on some neighbourhood of $x_iKK$.
Let $\Omega$ be the set of all
$\gamma\in C^\infty_c(G)$ such that $\|\gamma\|^L_{i,x_iKK}=\|h_i\gamma\|^L_{i,x_iKK}<1$
for all $i\in \N$ (with notation as in Definition~\ref{deflrmrtr}).
Then $\Omega$ is an open $0$-neighbourhood in $C^\infty_c(G)$
(cf.\ Lemma~\ref{lcsum}).
If $\beta$ was continuous, then we could find $0$-neighbourhoods
$V\sub C^r_c(G)$ and $W\sub C^\infty_c(G)$
such that $\beta(V\times W)\sub \Omega$.
There exists $s\in \N_0$ and $\tau >0$ such that
$\{\eta\in C^\infty_K(G)\colon \|\eta\|^L_s\leq \tau \}\sub W$.
Also, for each $i\in \N$ there exists $\sigma_i>0$
such that $\{\gamma\in C^r_{x_iK}(G)\colon \|\gamma\|^L_r\leq\sigma_i\}\sub V$.
Thus
$\beta(\gamma,\eta)\in \Omega$ and hence
\[
\|\gamma*\eta\|^L_{i,x_iKK}<1
\]
for all $\gamma\in C^r_{x_iK}(G)$
and $\eta\in C^\infty_K(G)$ such that $\|\gamma\|^L_r\leq \sigma_i$
and $\|\eta\|^L_s\leq \tau$.
Hence, using Lemmas~\ref{normtrans} and~\ref{efftra},
\[
\|\gamma*\eta\|^L_{i,KK} =\|\tau^L_{x_i^{-1}}(\gamma*\eta)\|_{i,x_iKK}^L
=\|(\tau^L_{x_i^{-1}}\gamma)*\eta\|^L_{i,x_iKK}<1
\]
for all $\gamma\in C^r_K(G)$
and $\eta\in C^\infty_K(G)$ such that $\|\gamma\|^L_r\leq \sigma_i$
and $\|\eta\|^L_s\leq \tau$. But this contradicts the following lemma.\,\Punkt
\begin{la}\label{notcs}
Let $G$ be a non-discrete Lie group, $K\sub G$ be a compact identity neighbourhood,
and $r,s\in \N_0$.
Then the convolution map
\begin{equation}\label{suppcts}
(C^r_K(G),\|.\|^L_r)\times (C^\infty_K(G),\|.\|^L_s)\to C^\infty_{KK}(G)
\end{equation}
is discontinuous,
if one uses the ordinary Fr\'{e}chet space topology on the right hand side,
but merely the two indicated norms on the left.
\end{la}
\begin{proof}
Suppose that the map (\ref{suppcts})
was continuous -- we shall derive a contradiction.
Let $\phi\colon P\to Q\sub \cg$ be a chart for $G$
around~$1$, such that $\phi(1)=0$, $P=P^{-1}$, $d\phi|_\cg=\id_\cg$
and $\phi(x^{-1})=-\phi(x)$ for all $x\in P$ (for example, a logarithmic chart).
After shrinking~$K$, we may assume that
$K=\phi^{-1}(A)$
for some compact $0$-neighbourhood
$A\sub Q$ with $[{-1},1]A\sub A$.
In particular, $K\sub P$.
Let $m>0$ be the dimension of $G$,
$\lambda_\cg$ be a Haar measure on~$(\cg,+)$,
and $\lambda_\cg|_Q$ be its restriction
to a measure on $(Q,\cB(Q))$.
Then the image measure
$\phi_*(\lambda_G|_P)$ is of the form $\rho \, d\lambda_\cg|_Q$
with a smooth function $\rho\colon Q\to \,]0,\infty[$.
Given $\gamma\in C^\infty_K(G)$,
let $\tilde{\gamma}:=\gamma\circ \phi^{-1} \in C^\infty_A(Q)$.
Then, for all $\gamma,\eta\in C^\infty_K(G)$,
\begin{equation}\label{ans1}
(\gamma*\eta)(0)=\int_Q \tilde{\gamma}(y)\tilde{\eta}({-y})\,\rho(y)\,d\lambda_\cg(y)\,.
\end{equation}
If $Y$ is a vector field on $G$ and $\theta:=d\phi\, \circ Y\circ \phi^{-1}\in C^\infty(Q,\R^m)$
its representative with respect to the chart~$\phi$,
then
\begin{equation}\label{ans2}
(D_Y.\gamma)\tilde{\;}=D_\theta.\wt{\gamma},
\end{equation}
where $D_Y(\gamma)$ is as in~(\ref{defnDX}), and $D_\theta.\wt{\gamma}:=d\wt{\gamma}\circ (\id_Q,\theta)$.
Choose $n\in 2\N$ so large that
$m+r+s+2-2n<0$.
Pick $h\in C^\infty_A(\cg)$ such that
$h\not=0$ and $h(x)=h(-x)$ for all $x\in \cg$.
There is $v\in A\setminus \{0\}$ such that $h(v)\not=0$.
Then $D_v^nh\not=0$.
To see this,
find $c>1$
such that $]{-c},c[\, v\sub Q$
but $]{-c},c[\, v\not\sub A$.
Then $g\colon ]{-c},c[\,\to \R$, $t\mto h(tv)$
is a compactly supported non-zero function,
whence $g^{(n)}(t_0)\not=0$ for some $t_0\in \,]{-c},c[$
and thus $D^n_vh(t_0v)\not=0$.
For $t\in \,]0,1]$,
define $\tilde{\gamma}_t,\tilde{\eta}_t \in C^\infty_{tA}(Q)\sub
C^\infty_A(Q)$ via
\[
\tilde{\gamma}_t(x):=t^{r+1}h(x/t)\quad\mbox{and}\quad
\tilde{\eta}_t(x):=t^{s+1}h(x/t)
\]
for $x\in Q$. Define $\gamma_t,\eta_t\in C^\infty_K(G)$
via $\gamma_t(x):=\tilde{\gamma}_t(\phi(x))$
and $\eta_t(x):=\tilde{\eta}_t(\phi(x))$
if $x\in P$, $\gamma_t(x):=\eta_t(x):=0$
if $x\in G\setminus K$.\\[2.5mm]
{\bf Claim:} \emph{$\gamma_t\to 0$ in $C^r_K(G)$
and $\eta_t\to 0$ in $(C^\infty_K(G),\|.\|_s^L)$ as $t\to 0$.
But $\|\gamma_t*\eta_t\|^{R,L}_{n,n}\to\infty$ as $t\to 0$,
whence $\gamma_t*\eta_t\not\to 0$ in $C^\infty_{KK}(G)$.}\\[2.5mm]
Therefore the map in (\ref{suppcts}) is not continuous, contradiction.\\[2.5mm]
To prove the claim, we first note that
\begin{equation}\label{nueq}
(D_{\xi_j}\ldots D_{\xi_1}.\tilde{\gamma}_t)(x)
=\sum_{i=1}^j t^{r+1-i}g_i(t,x)
\end{equation}
for $j\in \N$, $\xi_1\ldots, \xi_j\in C^\infty(Q,\cg)$ and $x\in Q$,
where
\[
g_j(t,x)=d^{(j)}h(x/t,\xi_1(x),\ldots, \xi_j(x))
\]
and $g_i(t,x)$ for $i<j$ is a sum of terms of the form
$d^{(i)} h (x/t,f_1(x),\ldots, f_i(x))$
with suitable smooth functions $f_1,\ldots,f_i\in C^\infty(Q,\cg)$.
This can be\linebreak
established by a straightforward induction,
using that application of $D_\xi$
to $x\mto d^{(i)} h (x/t,f_1(x),\ldots, f_i(x))$ yields\footnote{Recall
that $d^{(i)}h(x,.)\colon E^i\to F$
is $i$-linear (see, e.g., 
\cite{RES}, \cite{Ham}, \cite{Mil}).}
\begin{eqnarray*}
\lefteqn{\frac{1}{t}d^{(i+1)} h (x/t,f_1(x),\ldots, f_i(x), \xi(x)) }\qquad\qquad\\
& & \hspace*{-15mm}+d^{(i)} h (x/t,(D_\xi f_1)(x),\ldots, f_i(x))
+\cdots +
d^{(i)} h (x/t,f_1(x),\ldots, D_\xi(f_i)(x))
\end{eqnarray*}
for $\xi\in C^\infty(Q,\cg)$. A similar description (with $s$ in place of $r$)
can be given for $D_{\xi_j}\ldots D_{\xi_1}.\tilde{\eta}_t$.
We find it useful to abbreviate $h^{(i)}:=d^{(i)}h$ and\footnote{As
usual, for normed spaces $(E_1,\|.\|_1),\ldots,(E_i,\|.\|_i)$
and $(F,\|.\|_F)$
and a continuous $i$-linear map $B\colon E_1\times \cdots \times E_i\to F$,
we define $\|B\|_{ \op}$ as the supremum of $\|B(x_1,\ldots, x_i)\|_F$,
where $x_j\in E_j$ with $\|x_j\|_j\leq 1$
for $j=1,\ldots, i$.}\[
\|h^{(i)}\|_{ \op,\infty}:=\sup\{\|h^{(i)}(y,.)\|_{ \op }\colon y\in \cg\}\,.
\] 
Note that $h^{(i)}(x/t, f_1(x),\ldots, f_i(x))$
vanishes for $x$ outside $tA$ and hence for $x\not\in A$,
and that its norm is bounded by
\[
\|h^{(i)}\|_{\op,\infty } \,\|f_1\|_\infty\cdots \|f_i\|_\infty\,,
\]
irrespective of $t$ and $x$. A similar estimate is available for $g_j(t,x)$.
Also, $\|\tilde{\gamma}_t\|_\infty\leq \|h \|_\infty t^{r+1}$.
Hence, if $j\in \{0,\ldots, r\}$,
we can find $C>0$ such that
\[
\|D_{\xi_j}\ldots D_{\xi_1}.\tilde{\gamma}_t\|_\infty
\leq \sum_{i=1}^j t^{r+1-i}C\leq jtC\,.
\]
As a consequence,
\[
\max_{|\alpha|\leq r}\|\partial^\alpha \tilde{\gamma}_t\|_\infty\to 0
\]
as $t\to 0$ and thus $\tilde{\gamma}_t\to 0$ in $C^r_A(Q)$,
entailing that $\gamma_t\to 0$ in $C^r_K(G)\isom C^r_K(P)$.
Likewise, $\eta_t\to 0$ in $C^s_K(G)$,
whose topology can be described
by the norm $\|.\|^L_s$,
and thus $\eta_t\to 0$ in $(C^\infty_K(G),\|.\|_s^L)$.\\[2.5mm]
Next, let $X$ be the right invariant vector field on $G$
with $X(1)=v$, and $Y$ be the left invariant vector field
with $Y(1)=v$.
Let $\xi:=d\phi\circ X\circ \phi^{-1}$ and $\zeta:=d\phi\circ Y\circ\phi^{-1}$
be the local representatives.
By (\ref{ans1}) and (\ref{ans2}),
\[
(D_X^nD_Y^n(\gamma_t*\eta_t))(0)
=
(D_X^n\gamma_t* D_Y^n\eta_t)(0)
\! =\!
\int_Q (D_\xi^n\tilde{\gamma}_t)(y)(D_\zeta^n\tilde{\eta}_t)({-y})\,\rho(y)\,d\lambda_\cg(y).
\]
Write $(D_\xi^n\tilde{\gamma}_t)(x)=\sum_{i=1}^n t^{r+1-i}g_i(t,x)$
and
$(D_\zeta^n\tilde{\eta}_t)(x)=\sum_{j=1}^n t^{s+1-j}h_j(t,x)$ as in~(\ref{nueq}).
Then
\[
(D_X^n D_Y^n(\gamma_t*\eta_t))(0)=t^{m+r+s+2-2n}
\left(t^{-m}\int_Q g_n(t,y)h_n(t,-y)\,\rho(y)\,d\lambda_\cg(y)
+R(t)\right)
\]
where $R(t)$ is the sum of the terms $t^{2n-i-j}t^{-m}\int_Qg_i(t,y)h_j(t,-y)\rho(y)\,d\lambda_\cg(y)$
with $i,j\in \{1,\ldots, n\}$ and $(i,j)\not=(n,n)$.
For these $(i,j)$,
\[
\hspace*{-10mm}\left|t^{-m}\!\!\int_Qg_i(t,y)h_j(t,-y)\rho(y)\,d\lambda_\cg(y)\right|\hspace*{55mm}\vspace{-5.3mm}
\]
\begin{eqnarray*}
&=\!&\!
\left|t^{-m}\!\!\!\int_Qh^{(i)}(y/t, f_1(y),\ldots, f_i(y))h^{(j)}(-y/t, k_1(-y),\ldots, k_j(-y)))\rho(y)\,d\lambda_\cg(y)\right| \\
&=\!&\!
\left|\int_{Q/t}h^{(i)}(z, f_1(tz),\ldots, f_i(tz))h^{(j)}(-z, k_1(-tz),\ldots, k_j(-tz)))\rho(tz)\,d\lambda_\cg(z)\right| \\
&=\!&\!
\left|\int_A h^{(i)}(z, f_1(tz),\ldots, f_i(tz))h^{(j)}(-z, k_1(-tz),\ldots, k_j(-tz)))\rho(tz)\,d\lambda_\cg(z)\right| \\
& \leq\! &\! \|h^{(i)}\|_{\op,\infty}\|h^{(i)}\|_{\op,\infty}\|f_1\|_\infty\cdots\|f_i\|_\infty\|k_1\|_\infty\cdots\|k_j\|_\infty
\|\rho|_A\|_\infty\lambda_\cg(A)
\end{eqnarray*}
(using the substitution $y/t=z$ to obtain the second equality),
where the final estimate is independent of $t$. Since $2n-i-j\geq 1$
and thus $t^{2n-i-j}\to 0$ as $t\to 0$,
we deduce that $R(t)\to 0$. Similarly, substituting $z=y/t$, we get
\begin{eqnarray}
\lefteqn{t^{-m}\int_Q g_n(t,y)h_n(t,-y)\,\rho(y)\,d\lambda_\cg(y)}\notag\\
\hspace*{-25mm}&=&
t^{-m}\int_Q h^{(n)}(y/t,\xi(x),\ldots,\xi(x)))h^{(n)}(-y/t,\zeta(-y),\ldots,\zeta(-y))\,\rho(y)\,d\lambda_\cg(y)\notag\\
&=& \int_{Q/t} h^{(n)}(z,\xi(tz),\ldots,\xi(tz)))h^{(n)}(-z,\zeta(-tz),\ldots,\zeta(-tz))\,\rho(tz)\,d\lambda_\cg(z)\notag\\
&=& \int_A h^{(n)}(z,\xi(tz),\ldots,\xi(tz)))h^{(n)}(-z,\zeta(-tz),\ldots,\zeta(-tz))\,\rho(tz)\,d\lambda_\cg(z)\,\label{hereconv}
\end{eqnarray}
which tends to
\begin{eqnarray*}
\lefteqn{\int_A h^{(n)}(z,\xi(0),\ldots,\xi(0)))h^{(n)}(-z,\zeta(0),\ldots,\zeta(0))\,\rho(0)\,d\lambda_\cg(z)}\\
&=& \int_A h^{(n)}(z,v,\ldots,v))h^{(n)}(-z,v,\ldots,v)\,\rho(0)\,d\lambda_\cg(z)\\
&=& \rho(0) \int_A (D_v^nh(z))^2 \,d\lambda_\cg(z)=:a >0
\end{eqnarray*}
as $t\to 0$. Note that the integrand in (\ref{hereconv}) is continuous for $(t,y)\in [0,1]\times A$,
whence Lemma~\ref{ctspar} applies.
Since $R(t)\to 0$, there exists $\tau\in \,]0,1]$ such that $|R(t)|\leq a/2$ for all $t\in \,]0,\tau]$.
Then
$(D_X^n D_Y^n(\gamma_t*\eta_t))(0)\geq t^{m+r+s+2-2n}\frac{a}{2}$
for all $t\in \,]0,\tau]$, which tends to $\infty$ as $t\to 0$.
Hence also $\|\gamma_t*\eta_t\|^{R,L}_{n,n}\geq |(D_X^n D_Y^n(\gamma_t*\eta_t))(0)|$
goes to $\infty$ as $t\to 0$, and the claim is established.
\end{proof}
\section{Convolution on {\boldmath$\sigma$}-compact groups}\label{secsigma}
In this section, we complete the proof
of Theorem~A. As we have already
seen in Sections~\ref{aness} and \ref{bness}
that conditions (a) and (b) of the theorem
are necessary for continuity of~$\beta$,
it suffices to consider the case
that (a) and (b) are already satisfied,
and to show that $\beta$ is continuous if and only if
condition (c) of the theorem is satisfied.
In parallel,
we shall establish a result concerning discrete groups.
To formulate it, let $(G,r,s,t,b)$ be as in  the introduction.
If $G$ is discrete, then $C^p_c(G,H)=\bigoplus_{g\in G}H=:H^{(G)}$
with the locally convex direct sum topology, for each $p\in \N_0\cup\{\infty\}$
and locally convex space~$H$. We show:
\begin{prop}\label{discrca}
If $G$ is an infinite discrete group, then the map\linebreak
$\beta\colon E_1^{(G)}\times E_2^{(G)}\to F^{(G)}$, $\beta(\gamma,\eta):=\gamma *_b\eta$
is continuous if and only if $G$ is countable and $b\colon E_1\times E_2\to F$
satisfies product estimates.
\end{prop}
We need only prove 
Proposition~\ref{discrca} for countable~$G$
(as the discontinuity of~$\beta$
for uncountable~$G$ was already established
in Proposition~\ref{disco}).
\begin{la}\label{lalofin}
Let $G$ be a $\sigma$-compact, non-compact,
locally compact group
and $V\sub G$ be a compact identity neighbourhood.
Then there are sequences $(g_i)_{i\in \N}$ and $(h_j)_{j\in \N}$ in $G$,
such that the family $(g_iVh_jV)_{(i,j)\in \N^2}$ is locally finite.
\end{la}
\begin{proof}
Since $G$ is locally compact and $\sigma$-compact, there exists
a sequence $(K_i)_{i\in\N}$ of compact subsets of $G$ such that
$G=\bigcup_{i\in \N}K_i$ and $K_i\sub K_{i+1}^0$, for all $i\in \N$.
We may assume that $K_1=\emptyset$.
It suffices to find sequences $(g_i)_{i\in \N}$ and $(h_j)_{j\in\N}$ in $G$
such that
\begin{equation}\label{strg}
g_iVh_jV\cap K_{i\vee j}=\emptyset
\end{equation}
for all $i,j\in \N$, where $i\vee j$ denotes the maximum of $i$ and $j$.
Indeed, if $K\sub G$ is compact,
then $K\sub K_n$ for some $n\in\N$
and thus $K\cap g_iVh_jV=\emptyset$
unless $i,j\in \{1,\ldots, n-1\}$ (which is a finite set).
To find such sequences,
we make an arbitrary choice of $g_1, h_1\in G$.
Let $n\in \N$ now and assume that $g_i,h_j$ have been chosen
for $i,j\in \{1,\ldots, n\}$ such that (\ref{strg})
holds.
Then the subset\vspace{-1.3mm}
\[
P:=\bigcup_{j=1}^n \, K_{n+1}V^{-1}h_j^{-1}V^{-1}\vspace{-2mm}
\]
of $G$ is compact. As $G$ is non-compact,
we find $g_{n+1}\in G\setminus P$.
Also
$Q:=\bigcup_{i=1}^{n+1}\, V^{-1}g_i^{-1}K_{n+1}V^{-1}$
is compact, whence we find $h_{n+1}\in G\setminus Q$.
Then (\ref{strg}) holds for all $i,j \in \{1,\ldots,n+1\}$.
We need only check this if $i=n+1$ or $j=n+1$.
If $j=n+1$, then
$h_j=h_{n+1}\not\in Q$
and thus $g_iVh_jV\cap K_{n+1}=\emptyset$.
If $j\leq n$ and $i=n+1$, then $g_i=g_{n+1}\not\in P$
and hence $g_iVh_jV\cap K_{n+1}=\emptyset$.
\end{proof}
We shall use the seminorm $\|.\|_{p,L^1}$ on $C^0_c(G,E)$
(where $p$ is a continuous seminorm on~$E$), defined
via $\|\gamma\|_{p,L^1}:=\int_Gp(\gamma(y))\,d\lambda_G(y)$.
For each compact subset $K\sub G$,
we have $\|\gamma\|_{p,L^1}\leq \lambda_G(K)\|\gamma\|_{p,\infty}$
for all $\gamma\in C^0_K(G,E)$.
Hence $\|.\|_{p,L^1}$
is continuous on $C^0_K(G,E)$ and hence also on $C^0_c(G,E)$.\\[3mm]
\emph{Necessity of product estimates.}
Let $(G,r,s,t,b)$ be as in the introduction.\linebreak
Assume that~$G$ is not compact,
and asume that the conditions (a) and (b)
from Theorem~A are satisfied.\footnote{If $G$ is discrete,
these conditions are equivalent to countability of~$G$.}
Also, assume that $\beta$ is continuous.
Pick a relatively compact, open identity neighbourhood $V\sub G$.
By Lemma~\ref{lalofin}, there are sequences $(x_i)_{i\in \N}$
and $(y_j)_{j\in\N}$ in $G$ such that the family $(V_{i,j})_{(i,j)\in \N^2}$
of open sets $V_{i,j}:=x_iVy_jV$ is locally finite.
Pick $g_i,h_j\in C^r_c(G)$
such that $g_i,h_j\geq 0$,
$K_i:=\Supp(g_i)\sub x_iV$,
$L_j:=\Supp(h_j)\sub y_jV$,
and $\|g_i\|_{L^1}=\|h_j\|_{L^1}=1$.
Pick $h_{i,j}\in C^r_c(G)$
such that $h_{i,j}\geq 0$, $\Supp(h_{i,j})\sub V_{i,j}$
and $h_{i,j}|_{K_iL_j}=1$.
For $i,j\in \N$, let $p_{i,j}$ be a continuous seminorm on $F$.
Let $Z$ be the set of all
$\gamma\in C^t_c(G,F)$ such that
\[
(\forall i,j\in \N)\quad \|h_{i,j}\cdot \gamma\|_{p_{i,j},L^1}<1\,.
\]
Lemma~\ref{lcsum} entails that $Z$ is an open $0$-neighbourhood in
$C^t_c(G,F)$.
As $\beta$ is assumed continuous, there exist
open $0$-neighbourhoods $X\sub C^r_c(G,E_1)$
and $Y\sub C^s_c(G,E_2)$ such that $\beta(X\times Y)\sub Z$.
Using Lemma~\ref{scavect},
for each $i\in \N$ we find a continuous seminorm
$p_i$ on $E_1$
such that $g_i \wb{B}^{p_i}_1(0) \sub X$.
Likewise, for each $j\in \N$ there is a continuous
seminorm $q_j$ on $E_2$ such that $h_j\wb{B}^{q_j}_1(0) \sub Y$.
For $v\in \wb{B}^{p_i}_1(0)$ and $w\in \wb{B}^{q_j}_1(0)$,
we then have $\gamma:=\beta(g_i v,h_j w)\in Z$,
and thus
$\|\gamma\|_{p_{i,j},L^1}
= \|h_{i,j}\cdot \gamma\|_{p_{i,j},L^1}<1$
(noting that $\Supp(\gamma)\sub K_iL_j$
on which $h_{i,j}=1$).
Therefore
\begin{eqnarray*}
1 & > & \|\gamma\|_{p_{i,j},L^1}\; =\; \int_G p_{i,j}(\gamma(x))\,d\lambda_G(x)\\
&=& \int_G \int_G p_{i,j}(b(g_i(y)v,h_j(y^{-1}x)w)\,d\lambda_G(y)\, d\lambda_G(x)\\
&=&  p_{i,j}(b(v,w))\int_G \int_G g_i(y)h_j(y^{-1}x)\,d\lambda_G(y)\, d\lambda_G(x)\\
&=& p_{i,j}(b(v,w))\, \|g_i\|_{L^1}\|h_j\|_{L^1}
\;=\;p_{i,j}(b(v,w))\, .
\end{eqnarray*}
Hence $b(\wb{B}^{p_i}_1(0)\times \wb{B}^{q_j}_1(0))\sub \wb{B}^{p_{i,j}}_1(0)$,
entailing that $p_{i,j}(b(x,y))\leq p_i(x)q_j(y)$ for all $i,j\in \N$.
Thus $b$ satisfies product estimates.
\begin{numba}\label{reuseelwh}
\emph{Sufficiency of product estimates.}
As before, let $(G,r,s,t,b)$ be as in the introduction,
and assume that conditions (a) and (b)
from Theorem~A are satisfied.
Also, assume that $b$ satisfies product estimates
(condition\,(c)).
We show that $\beta$ is continuous.
To this end, let $(h_i)_{i\in \N}$
be a locally finite partition of unity on $G$
(smooth if $G$ is a Lie group; continuous
otherwise),
with compact supports $K_i:=\Supp(h_i)$.
For all $i,j\in \N$,
the convolution map $f_{i,j}\colon C^r_{K_i}(G,E_1)\times
C^s_{K_j}(G,E_2)\to C^t_c(G,F)$ associated with
$b$ is then continuous (see Lemmas~\ref{convcts}
and \ref{diffbar}).
We claim that the hypotheses of \cite[Corollary~2.5]{Glo}
are satisfied.
If this is true, then the bilinear map
\[
f\colon \bigoplus_{i\in \N}C^r_{K_i}(G,E_1)\times \bigoplus_{j\in \N}C^s_{K_j}(G,E_2)\to C^t_c(G,F)
\]
taking $((\gamma_i)_{i\in \N},(\eta_j)_{j\in \N})$ to $\sum_{(i,j)\in \N^2}f_{i,j}(\gamma_i,\eta_j)$
is continuous (by the latter corollary).
Since also the linear map
$\Phi\colon C^r_c(G,E_1)\to\bigoplus_{i\in\N} C^r_{K_i}(G,E_1)$,
$\gamma\mto (h_i\cdot \gamma)_{i\in \N}$
and the analogous map
$\Psi\colon C^s_c(G,E_2)\to\bigoplus_{i\in\N} C^s_{K_i}(G,E_2)$
are continuous,
we deduce that $\beta=f\circ (\Phi\times \Psi)$ is continuous.

To prove the claim, let $Q_{i,j}$ be continuous seminorms on $C^t_c(G,F)$
for all $i,j\in \N$.
If $t<\infty$, choose $k,\ell\in \N_0$
with $k\leq r$, $\ell\leq s$ and $k+\ell=t$.
If $t=r=s=\infty$, let $k:=\ell:=0$.
If $i<j$, then there exists a continuous seminorm $P_{i,j}$ on $F$
and $s_{i,j}\in \N_0$ such that $s_{i,j}\leq s$ and
\[
(\forall \gamma\in C^t_{K_iK_j}(G,F))\quad Q_{i,j}(\gamma)\leq \|\gamma\|^{R,L}_{k,s_{i,j},P_{i,j}}
\]
(see Lemma~\ref{enoughsn}).
If $i\geq  j$,
there exists a continuous seminorm $P_{i,j}$ on $F$
and $r_{i,j}\in \N_0$ such that $r_{i,j}\leq r$ and
$Q_{i,j}(\gamma)\leq \|\gamma\|^{R,L}_{r_{i,j},\ell,P_{i,j}}$
for all $\gamma\in C^t_{K_iK_j}(G,F)$.
By hypothesis~(c),
there are continuous seminorms $p_i$ on $E_1$ and $q_j$ on $E_2$
such that $P_{i,j}(b(x,y))\leq p_i(x)q_j(y)$
for all $i,j\in \N$ and all $x\in E_1$, $y\in E_2$.
For $i,j\in \N$, let $P_i:=\|.\|_{k,p_i}^R$
and $Q_j:=\|.\|_{\ell,q_j}^L$.
For $i,j\in\N$ with $i<j$,
let $q_{i,j}:=\lambda_G(K_i)\|.\|_{s_{i,j},q_j}^L$.
Then
\begin{equation}\label{gclaim1}
Q_{i,j}(\gamma*_b\eta)\leq
\|\gamma*_b\eta\|^{R,L}_{k,s_{i,j},P_{i,j}}\leq
\|\gamma\|_{k,p_i}^R\|\eta\|_{s_{i,j},q_j}^L\lambda_G(K_i)
=P_i(\gamma)q_{i,j}(\eta) 
\end{equation}
for all $\gamma\in C^r_{K_i}(G,E_1)$ and
$\eta\in C^s_{K_j}(G,E_2)$,
using Lemma~\ref{lasimpe}.
If $i,j\in\N$ with $i\geq j$,
let $p_{i,j}:=\lambda_G(K_i)\|.\|_{r_{i,j},p_i}^R$.
For $\gamma\in C^r_{K_i}(G,E_1)$ and
$\eta\in C^s_{K_j}(G,E_2)$,
\begin{equation}\label{gclaim2}
Q_{i,j}(\gamma*_b\eta)\leq
\|\gamma*_b\eta\|^{R,L}_{r_{i,j},\ell,P_{i,j}}\leq
\|\gamma\|_{r_{i,j},p_i}^R\|\eta\|_{\ell,q_j}^L\lambda_G(K_i)
= p_{i,j}(\gamma)Q_j(\eta)
\end{equation}
(using Lemma~\ref{lasimpe} again).
By (\ref{gclaim1}) and (\ref{gclaim2}), the claim is established.\vspace{3mm}\Punkt
\end{numba}
\section{Convolution of \boldmath{$C^r$}-maps and {\boldmath$C^s_c$}-maps}\label{secfin}
\begin{prop}\label{nospp}
Let $(G,r,s,t,b)$ be as in the introduction, and
\begin{eqnarray*}
\beta_b\colon C^r_c(G,E_1)\times C^s(G,E_2) & \to & C^t(G,F)\;\;\mbox{and}\\
\theta_b \colon C^r(G,E_1)\times C^s_c(G,E_2) & \to & C^t(G,F)
\end{eqnarray*}
be the convolution maps taking $(\gamma,\eta)$ to $\gamma*_b\eta$.
Then $\beta_b$ and $\theta_b$ are hypocontinuous.
The map $\beta_b$ is continuous if and only if $G$ is compact.
Likewise, $\theta_b$ is continuous if and only if $G$ is compact.
Moreover, the convolution maps
\begin{eqnarray*}
\beta_K\colon C^r_K(G,E_1)\times C^s(G,E_2) & \to & C^t(G,F)\;\;\mbox{and}\\
\theta_K \colon C^r(G,E_1)\times C^s_K(G,E_2) & \to & C^t(G,F)
\end{eqnarray*}
taking $(\gamma,\eta)$ to $\gamma*_b\eta$
are continuous, for each compact subset $K\sub G$.
\end{prop}
\begin{proof}
Since $\theta_b(\gamma,\eta)=\beta_{b^\vee}(\eta^*,\gamma^*)^*$
for all $(\gamma,\eta)\in C^r(G,E_1)\times C^s_c(G,E_2)$
and $\beta_b(\gamma,\eta)=\theta_{b^\vee}(\eta^*,\gamma^*)^*$
for all $(\gamma,\eta)\in C^r_c(G,E_1)\times C^s(G,E_2)$,
where the involutions denoted by $*$ are continuous linear maps
(Lemma~\ref{invoaut}), the assertions concerning~$\theta_b$ follow
if we can establish those concerning~$\beta:=\beta_b$.

We first show that $\beta_K$
is continuous for each compact subset $K\sub G$.
To this end, recall that the topology on $C^t(G,F)$ is initial with respect to
the linear maps $\rho_W\colon C^t(G,F)\to C^t(W,F)$, $\gamma\mto\gamma|_W$,
for $W$ ranging through the relatively compact, open subsets of~$G$
(cf.\ \cite[Lemma~4.6]{ZOO}).
%
%
Since $K^{-1}\wb{W}$ is compact, there exists $h\in C^\infty(G)$
with compact support $L:=\Supp(h)$,
such that $h|_{K^{-1}\wb{W}}=1$.
For $\gamma\in C^r_K(G,E_1)$ and $\eta\in C^s(G,E_2)$, we have
$(\gamma*_b\eta)(x)=\int_Gb(\gamma(x),\eta(y^{-1}x))\,d\lambda_G(y)
=\int_K b(\gamma(x),\eta(y^{-1}x))\,d\lambda_G(y)
=\int_K b(\gamma(x),(h\cdot \eta)(y^{-1}x))\,d\lambda_G(y)
=\int_G b(\gamma(x),(h\cdot \eta)(y^{-1}x))\,d\lambda_G(y)=$\linebreak
$(\gamma*_b(h\cdot \eta))(x)$
for all $x\in W$ and hence
\[
\rho_W(\beta_K(\gamma,\eta))=
\rho_W(\beta_K(\gamma,h\cdot \eta))
=\rho_W(\mu(\gamma,h\cdot \eta))\,,
\]
using the convolution
$\mu\colon C^r_K(G,E_1)\times C^s_L(G,E_2)\to C^t_{KL}(G,F)\sub C^t(G,F)$,
which is continuous by Lemmas~\ref{convcts} and \ref{diffbar}.
Since also the multiplication operator
$m_h\colon C^s(G,E_2)\to C^s_L(G,E_2)$, $\eta\mto h\cdot \eta$ is continuous
(cf.\ \cite[Lemma~3.9 and Proposition~3.10]{GCX}),
$\rho_W\circ \beta_K$ and hence $\beta_K$
is continuous.

If $\gamma\in C^r_c(G,E_1)$,
then $\beta_b(\gamma,.)=\beta_{\Supp(\gamma)}(\gamma,.)\colon C^s(G,E_2)\to C^t(G,F)$
is continuous.
For $\eta\in C^r(G,E_2)$,
the map $\beta_b(.,\eta)\colon C^r_c(G,E_1)\to C^t(G,F)$ is linear
and $\beta_b(.,\eta)|_{C^r_K(G,E_1)}=\beta_K(.,\eta)\colon
C^r_K(G,E_1)\to C^t(G,F)$ is continuous for each compact set $K\sub G$.
Hence, since $C^r_c(G,E_1)=\dl\,C^r_K(G,E_1)$\vspace{-.3mm}
as a locally convex space,
$\beta_b(.,\eta)$ is continuous.
Thus $\beta_b$ is separately continuous.

If $B\sub C^r_c(G,E_1)$ is a bounded set, then
$B\sub C^r_K(G,E_1)$
for some compact set $K\sub G$ (Lemma~\ref{towardsh}\,(c)),
and thus $\beta|_{B\times C^s(G,E_2)}=\beta_K|_{B\times C^s(G,E_2)}$
is continuous.
Hence $\beta_b$ is hypocontinuous in the first argument (Remark~\ref{easyhypo}).

To see that $\beta_b$ is hypocontinuous in the second argument,
let $(\lambda_i)_{i\in I}$ be a family of linear maps $\lambda_i\colon F\to F_i$
to Banach spaces~$F_i$, such that the topology on $F$ is initial
with respect to this family.
Then the topology on $C^t(G,F)$ is initial with respect to the mappings
$C^t(G,\lambda_i)$ for $i\in I$ (see \cite[Lemma~4.14]{ZOO} for manifolds;
cf.\ \cite[3.4.6]{Eng} for topological spaces).
Hence, by Lemma~\ref{basichypo}\,(c),
we need only
show that each of the maps $C^t(G,\lambda_i)\circ \beta_b=\beta_{\lambda_i\circ b}$\linebreak
is hypocontinuous in the second argument. We may therefore assume now that~$F$
is a Banach space. Then there exist continuous linear mappings\linebreak
$\psi_1\colon E_1\to F_1$
and $\psi_2\colon E_2\to F_2$ to suitable
Banach spaces $F_1$ and $F_2$, and a continuous bilinear map
$c\colon F_1\times F_2\to F$ such that $c\circ (\psi_1\times \psi_2)=b$.
Since $\beta_b=\beta_c\circ (C^r_c(G,\psi_1)\times C^s(G,\psi_2))$,
we need only show that $\beta_c$ is hypocontinuous (see Lemma~\ref{basichypo}\,(b)).
We may therefore assume that all of $E_1$, $E_2$, and $F$ are Banach spaces.
Then $C^r_K(G,E_1)$ is a Fr\'{e}chet space for each\linebreak
compact subset $K\sub G$,
and hence barrelled.
Hence also $C^r_c(G,E_1)$ is\linebreak
barrelled,
like every locally convex direct limit of barrelled spaces \cite[II.7.2]{SaW}.
As the first factor of its domain is barrelled,
the separately continuous bilinear map
$\beta_b\colon C^r_c(G,E_1)\times C^s(G,E_2)\to C^t(G,F)$
is hypocontinuous in the second argument
\cite[III.5.2]{SaW}. As we already established its hypocontinuity in
the first argument, $\beta_b$ is hypocontinuous.

Finally, we show that $\beta_b$ (and hence also $\theta_b$) fails to be continuous if~$G$ is not compact.
Pick $u\in E_1$, $v\in E_2$ such that $w:=\beta(u,v)\not=0$.
Let $K\sub G$ be a compact identity neighbourhood
and~$p$ be a continuous seminorm on~$F$
such that $p(w)>0$.
Then $W:=\{\gamma\in C^t(G,F)\colon \gamma(K)\sub B^p_1(0)\}$
is an open $0$-neighbourhood in $C^t(G,F)$.
To see that $\beta_b$ is not continuous,
let $U\sub C^r_c(G,E_1)$
and $V\sub C^s(G,E_2)$ be any $0$-neighbourhoods.
Let $(U_i)_{i\in I}$
be a locally finite cover of~$G$ be relatively compact, open sets.
Since the topology on $C^s(G,E_2)$
is initial with respect to the restriction mappings $\rho_i:$\linebreak
$C^s(G,E_2)\to C^s(U_i,E_2)$, $\gamma\mto \gamma|_{U_i}$
(cf.\ \cite[Lemma~4.12]{ZOO}),
we find a finite subset $I_0\sub I$ and $0$-neighbourhoods
$Q_i\sub C^s(U_i,E_2)$ for $i\in I_0$ such that
\begin{equation}\label{usefq}
\bigcap_{i\in I_0}\rho_i^{-1}(Q_i)\sub V\,.
\end{equation}
Since $K:=\bigcup_{i\in I_0}\wb{U_i}$ is compact
and we assume that $G$ is not compact,
the open set $G\setminus K$ is non-empty.
We pick $\eta\in C^\infty_c(G)$ such that $\eta\geq 0$,
$\eta\not=0$ and $\Supp(\eta)\sub G\setminus K$.
Then $ \eta_a:=a \gamma v\in V$ for each $a>0$ (by (\ref{usefq})).
Define $\gamma_a \in C^r_c(G,E_1)$ via $\gamma_a(x):=\frac{1}{a}\eta(x^{-1}) u$.
Since~$U$ is a $0$-neighbourhood,
there is $a_0>0$ such that $\gamma_a\in U$ for all $a\geq a_0$.
Then
\[
p((\gamma_a*_b \eta_{a^2})(1))=a p(w)\int_G\eta(y^{-1})\eta(y^{-1})\,d\lambda_G(y)
=a p(w)\|\eta\|_{L^2}^2\,,
\]
where the right hand side can be made $>1$ for large~$a$.
Thus $\gamma_a *_b\eta_{a^2}\not\in W$
although $\gamma_a\in U$ and $\eta_{a^2}\in V$.
Hence $\beta_b(U\times V)\not\sub W$.
Since~$U$ and~$V$ were arbitrary, $\beta_b$ is not continuous.\vspace{-6mm}
\end{proof}
\appendix
\section{Background on vector-valued integrals}\label{appwint}\vspace{-2.5mm}
If $E$ is a locally convex space, $(X,\Sigma,\mu)$
a measure space~\cite{Bau}
and $\gamma\colon X\to E$
a function, we call a (necessarily unique)
element $v\in E$ the \emph{weak integral of $\gamma$ with respect to $\mu$}
(and write $\int_X\gamma(x)\,d\mu(x):=v$)
if $\lambda\circ \gamma\colon X\to\R$
is $\mu$-integrable for each $\lambda\in E'$
and $\lambda(v)=\int_X\lambda(\gamma(x))\,d\mu(x)$.
If $p$ is a continuous seminorm on~$E$,
using the Hahn-Banach Extension Theorem
one finds that
\begin{equation}\label{intabsch}
p\left( \int_X\gamma(x)\,d\mu(x)\right) \leq \mu(X)\|\gamma\|_{p,\infty}\,.
\end{equation}
\begin{la}\label{intex}
Let $(E,\|.\|)$ be a locally convex space,
$X$ be a locally compact space,
$\mu$ be a Borel measure on~$X$ $($see \emph{\ref{measet})}
and $\gamma\colon X\to E$ be a continuous mapping
with compact support~$K$.
If~$K$ is metrizable, assume that~$E$ is sequentially complete
or satisfies the metric convex compactness property;
if~$K$ is not metrizable, assume that~$E$
satisfies the convex compactness property.
Then the weak integral $\int_X\gamma(x)\,d\mu(x)$
exists in~$E$.
\end{la}
\begin{proof}
See \cite[1.2.3]{Her} for the first case,
\cite[3.27]{RFA} for the two others.
\end{proof}
The next two lemmas can be proved exactly as \cite[Proposition~3.5]{Bil}.
\begin{la}\label{ctspar}
Let $E$ be a locally convex space, $X$ be a topological space, $\mu$ a Borel measure on
a compact space~$K$, and $f\colon X \times K\to E$ be a continuous map.
Assume that the weak integral $g(x):=\int_Kf(x,a)\,d\mu(a)$ exists for each $x\in X$.
Then $g\colon X\to E$ is continuous.
\end{la}
\begin{la}\label{diffpar}
In the situation of~{\rm\ref{ctspar}}, assume
that $n\in \N$, $r\in \N_0\cup\{\infty\}$,
$X \sub \R^n$ is open,
the partial derivatives $\partial^\alpha_1f(x,a)$
of $f$ with respect to the variables in~$X$
exist for all $\alpha\in\N_0^n$ with $|\alpha|\leq r$,
and define continuous maps $\partial^\alpha_1f\colon X\times K\to E$.
Also, assume that the weak integrals
$\int_K \partial^\alpha_1f(x,a)\,d\mu(a)$ exist in~$E$ for all~$\alpha$ as before.
Then
$g\colon X\to E$, $x\mto \int_Kf(x,a)\,d\mu(a)$
is~$C^r$,
and $\partial^\alpha g(x)\!=\! \int_K \partial^\alpha_1f(x,a)\,d\mu(a)$
for all $\alpha\! \in\!  \N_0^n$ with $|\alpha|\leq r$ and $x\in K$.\vspace{-1.7mm}
\end{la}
\section{Hypocontinuous bilinear maps}\label{hypoapp}\vspace{-2mm}
{\bf Hypocontinuity.}
As a special case of more general concepts,
we call a bilinear map $\beta\colon E_1\times E_2\to F$
between locally convex spaces
\emph{hypocontinuous} if the following conditions are satisfied:
\begin{itemize}
\item[(H1)]
For each $0$-neighbourhood $V\sub F$
and bounded set $B_1\sub E_1$,
there exists a $0$-neighbourhood $U\sub E_2$
such that $\beta(B_1\times U)\sub V$; and:
\item[(H2)]
For each $0$-neighbourhood $V\sub F$
and bounded set $B_2\sub E_2$,
there exists a $0$-neighbourhood $U\sub E_1$
such that $\beta(U\times B_2)\sub V$.
\end{itemize}
In this case, $\beta$ is separately continuous (as $B_1$, $B_2$ can be chosen as singletons).
If~$\beta$ is separately continuous and satisfies (H1) (resp., (H2)),
we say that~$\beta$ is hypocontinuous in its first (resp., its second) argument.
\begin{rem}\label{easyhypo}
It is well known that a separately continuous
bilinear map $\beta\colon E_1\times E_2\to F$ between locally convex spaces
is hypocontinuous in its second argument
if and only if its restrictions
$\beta|_{E_1\times B}\colon E_1\times B\to F$ are continuous
for all bounded subsets $B\sub E_2$
(see, e.g.,  \cite[Proposition~16.8]{Hyp}; cf.\ Proposition~4 in
\cite[Chapter III, \S5, no.\,3]{Bou}).
\end{rem}
Simple observations
concerning hypocontinuous bilinear maps will be useful.
\begin{la}\label{basichypo}
\begin{itemize}
\item[\rm(a)]
Let $\beta\colon E_1\times E_2\to F$ be a bilinear map between locally convex spaces
which is hypocontinuous in its second argument.
Let $\Lambda\colon F\to H$ be a continuous linear map to a locally convex space~$H$.
Then also $\Lambda \circ\beta\colon E_1\times E_2\to H$ is hypocontinuous in its second argument.
\item[\rm(b)]
Let $\beta\colon E_1\times E_2\to F$ be a bilinear map between locally convex spaces
which is hypocontinuous in its second argument.
Let $H_1$, $H_2$ be locally convex spaces and
$\psi_1\colon H_1\to E_1$, $\psi_2\colon H_2\to E_2$ be continuous linear maps.
Then also $\beta\circ(\psi_1\times \psi_2) \colon H_1\times H_2\to F$ is hypocontinuous in its second
argument.
\item[\rm(c)]
Let $E_1$, $E_2$ and~$F$ be locally convex spaces.
If the topology on~$F$ is initial with respect to a family $(\Lambda_i)_{i\in I}$ of linear maps
$\Lambda_i \colon F\to F_i$ to locally convex spaces~$F_i$,
then a bilinear map $\beta\colon E_1\times E_2\to F$ is hypocontinuous in its second argument
if and only if $\Lambda_i\circ \beta$
is hypocontinuous in its second argument, for each $i\in I$.
\item[\rm(d)]
Let $(E_i)_{i\in I}$ and $(F_j)_{j\in J}$ be families of locally convex spaces
and let\linebreak
$\beta\colon \big(\bigoplus_{i\in I}E_i\big)\times \big(\bigoplus_{j\in J}F_j\big)\to H$
be a bilinear map
to a locally convex space~$H$.
Then~$\beta$ is hypocontinuous
in its second argument if and only if $\beta_{i,j}:=\beta|_{E_i\times F_j}\colon E_i\times F_j\to H$
is hypocontinuous in its second argument, for all $(i,j)\in I\times J$.
\item[\rm(e)]
If $E_1$, $E_2$ are locally convex spaces and $\beta\colon E_1\times E_2\to F$ is a continuous
bilinear map to a Fr\'{e}chet space~$F$, then there exist continuous linear maps
$\psi_1\colon E_1\to F_1$ and  
$\psi_2\colon E_2\to F_2$ to Fr\'{e}chet spaces $F_1,F_2$
and a continuous bilinear map $\theta\colon F_1\times F_2\to F$
such that $\beta=\theta\circ (\psi_1\times \psi_2)$.
\end{itemize}
\end{la}
\begin{proof}
(a) $\Lambda\circ \beta$ is separately continuous,
and $\Lambda\circ\beta|_{E_1\times B}$ is continuous
for each bounded subset $B\sub E_2$. Hence $\Lambda\circ\beta$
is hypocontinuous in its second argument (see Remark~\ref{easyhypo}).

(b) The composition $\gamma:=\beta\circ(\psi_1\times \psi_2)$ is separately continuous.
If $B\sub H_2$ is bounded, then $\psi_2(B)$ is bounded in~$E_2$,
entailing that
$\gamma|_{H_1\times B}=\beta|_{E_1\times \psi_2(B)}\circ \psi_1\times (\psi_2|_B)$ is continuous.
Hence $\gamma$ is hypocontinuous in its second argument (using
Remark~\ref{easyhypo} again).

(c) If $x\in E_1$, then $\beta(x,.)\colon E_2\to F$ is continuous if and only if\linebreak
$\Lambda_i\circ \beta(x,.)\colon E_2\to F_i$
is continuous for each $i\in I$.
Likewise, $\beta(.,y)$ is continuous for $y\in E_2$
if and only if $\Lambda_i\circ \beta(.,y)$ is continuous for each~$i$.
If $B\sub E_2$ is bounded, then $\beta|_{E_1\times B}$
is continuous if and only if $\Lambda_i\circ\beta|_{E_1\times B}$
is continuous for each $i\in I$.
The assertion follows with Remark~\ref{easyhypo}.

(d) Write $E:=\bigoplus_{i\in I}E_i$,
$F:=\bigoplus_{j\in J}F_j$.
For $i\in I$, let $\lambda_i\colon E_i\to E$ be the usual
embedding. Also, let $\mu_j\colon F_j\to F$ be the embedding for $j\in J$.
If $\beta$ is hypocontinuous in its second argument,
then so is $\beta_{i,j}=\beta\circ (\lambda_i\times \mu_j)$, by~(b).

Conversely, assume that each $\beta_{i,j}$ is hypocontinuous in its second argument.
If $x=(x_i)_{i\in I}\in E$, then $x\in \bigoplus_{i\in I_0}E_i$
for some finite set $I_0\sub I$. The linear map
$\beta(x,.)\colon F\to H$ is continuous on~$F_j$ for each $j\in J$
(as it coincides with $\sum_{i\in I_0}\beta_{i,j}(x_i,.)$ there).
Hence $\beta(x,.)$ is continuous (by the universal property of the locally convex direct sum).
Likewise, $\beta(.,y)\colon E\to H$ is continuous for each $y\in F$,
and thus $\beta$ is separately continuous.

Let $B\sub \bigoplus_{j\in J}F_j$ be a bounded set and $U\sub H$ be an absolutely convex $0$-neighbourhood.
Then $B\sub\bigoplus_{j\in J_0}F_i=:X$ for some finite subset $J_0\sub J$ \cite[II.6.3]{Jar}.
Let $n$ be the number of elements of~$J_0$.
Excluding only trivial cases, we may assume that $n\geq 1$.
For $j\in J_0$, let $\pi_j\colon X\to F_j$ be the projection onto~$F_j$.
Then $B_j:=\pi_j(B)$ is bounded in~$F_j$.
For each $i\in I$, using the hypocontinuity of $\beta_{i,j}$
we now find a convex $0$-neighbourhood $V_{i,j}\sub E_i$
such that $\beta_{i,j}(V_{i,j}\times B_j)\sub \frac{1}{n}U$.
Set $V_i:=\bigcap_{j\in J_0}V_{i,j}$. Then
$\beta(V_i\times B_j)=\beta_{i,j}(V_i\times B_j)\sub \frac{1}{n}U$
and hence, using that $B\sub \sum_{j\in J_0}B_j$,
\[
\beta(V_i\times B)\sub \sum_{j\in J_0}\beta_{i,j}(V_i\times B_j)\sub \sum_{j\in J_0}\frac{1}{n}U\sub U\,.
\]
Now $V:=\conv(\bigcup_{i\in I}V_i)$ is a $0$-neighbourhood in~$E$.
As every $V_i$ is convex,
for each $x\in V$ there are a finite set $I_0\sub I$, elements $x_i\in V_i$ for $i\in I_0$,
and $t_i\geq 0$ for $i\in I_0$ with $\sum_{i\in I_0}t_i=1$ and $x=\sum_{i\in I_0}t_i x_i$.
Hence, for each $y\in B$,
\[
\beta(x,y)=\sum_{i\in I_0}t_i\beta(x_i,y)
\in \sum_{i\in I_0}t_iU\sub U\,.
\]
Thus $\beta(V\times B)\sub U$.
Hence $\beta$ is hypocontinuous in its second argument.

(e) Let $(s_n)_{n\in \N}$ be a sequence of continuous
seminorms on~$F$ defining its locally convex vector topology.
For each $n\in \N$, we then find continuous seminorms~$P_n$ on~$E_1$
and~$Q_n$ on~$E_2$ such that $s_n(\beta(x,y))\leq P_n(x)Q_n(y)$
for all $x\in E_1$, $y\in E_2$.
Then $N_1:=\{x\in E_1\colon (\forall n\in \N)\; P_n(x)=0\}$
is a vector subspace of $E_1$,
and  $N_2:=\{y\in E_2\colon (\forall n\in \N)\; Q_n(y)=0\}$
a vector subspace of~$E_2$.
We equip $E_1/N_1$
with the vector topology defined by the sequence
$(p_n)_{n\in \N}$ of seminorms given by
$p_n(x+N_1):=P_n(x)$,
and let~$F_1$ be the completion of $E_1/N_1$.
Likewise, $F_2$ denotes the completion of $E_1/N_2$,
equipped with the seminorms~$q_n$ obtained from the~$Q_n$.
If $x\in N_1$ and $y\in E_2$,
then $s_n(\beta(x,y))=0$ for each $n\in \N$ and thus
$\beta(x,y)=0$. Likewise,
$\beta(x,y)=0$ for all $x\in E_1$ and $y\in N_2$.
As a consequence,
$B \colon (E_1/N_1)\times (E_2/N_2)\to F$,
$B (x+N_1,y+N_2):=\beta(x,y)$ is a well-defined
bilinear map,
which is continuous as
$s_n(B(x+N_1,y+N_2))=s_n(\beta(x,y))\leq P_n(x)Q_n(y)=p_n(x+N_1)q_n(y+N_2)$.
Since~$F$ is complete, $B$ extends to a continuous bilinear map
$\theta\colon F_1\times F_2\to F$
(cf.\ Theorem~1 in \cite[III, \S6, no.\,5]{BTG}).
Let $\psi_1\colon E_1\to F_1$
be the composition of inclusion map $E_1/N_1\to F_1$
and the canonical map $E_1\to E_1/N_1$.
Define $\psi_2\colon E_2\to F_2$ analogously.
Then $\beta=\theta\circ (\psi_1\times \psi_2)$ indeed.\vspace{-7mm}
\end{proof}
\section{Proofs for Sections~\ref{secprel} to \ref{secmeasset}}\label{appproofs}\vspace{-2mm}
{\bf Proof of Lemma~\ref{sumemb}.}
Since $\lambda:=\oplus_{i\in I}\lambda_i$ is linear and continuous on each $E_i$,
it is continuous (by the universal property of the locally convex direct
sum). Moreover, $\lambda$ is injective since each $\lambda_i$ is injective.
To see that~$\lambda$ is an embedding,
let $U\sub \bigoplus_{i\in I}E_i=:E$ be a $0$-neighbourhood.
By Remark~\ref{semisum},
there is a continuous seminorm~$p$ on~$E$,
of the form $p(x)=\sum_{i\in I}p_i(x_i)$ with continuous seminorms~$p_i$ on~$E_i$,
such that $B^p_1(0)\sub U$.
Since $\lambda_i$ is an embedding,
there exists a continuous seminorm~$q_i$ on~$F_i$
such that $\lambda_i^{-1}(B^{q_i}_1(0))\sub B^{p_i}_1(0)$.
Then $q(y):=\sum_{i\in I}q_i(y_i)$ defines a continuous
seminorm on $\bigoplus_{i\in I}F_i$. We now show that
$\lambda(E)\cap B^q_1(0)\sub
\lambda(B^p_1(0))\sub \lambda(U)$
(whence $\lambda(U)$ is a $0$-neighbourhood in $\lambda(E)$ and hence~$\lambda$ open onto
its image -- as required).
In fact,
$\lambda_i(E_i)\cap B^{q_i}_1(0)\sub \lambda_i(B^{p_i}_1(0))$.
Hence $\lambda(E)\cap B^q_1(0)=\{y\in \bigoplus_{i\in I}\lambda_i(E_i)\colon
\sum_{i\in I}q_i(y_i)<1\}=\conv\bigcup_{i\in I}(\lambda_i(E_i)\cap B^{q_i}_1(0))\sub
\conv\bigcup_{i\in I}\lambda_i(B^{p_i}_1(0))\sub \lambda(B^p_1(0))$.\,\vspace{2mm}\Punkt

\noindent
{\bf Proof of Lemma~\ref{lcsum}.}
If $K\sub M$ is compact, then $F:=\{j\in J\colon K\cap K_j\not=\emptyset\}$
is a finite set, and
$\Phi(C^r_K(M,E))\sub \bigoplus_{j\in F} C^r_{K_j}(M,E)$,
identifying the right hand side with a topological vector subspace
of
$\bigoplus_{j\in J}C^r_{K_j}(M,E)$ in the usual way.
Since $\bigoplus_{j\in F} C^r_{K_j}(M,E)\isom \prod_{j\in F} C^r_{K_j}(M,E)$
as a topological vector space,
the restriction $\Phi_K$ of $\Phi$ to $C^r_K(M,E)$ will
be continuous if we can show that all of its components
with values in $C^r_{K_j}(M,E)$ are continuous, for $j\in F$.
But these are the multiplication operators $C^r_K(M,E)\to C^r_{K_j}(M,E)$, $\gamma\mto h_j\cdot \gamma$,
whose continuity is well-known (cf.\ \cite[Proposition~3.10]{GCX}).

If the $h_j$ form a partition of unity, let $S\colon \bigoplus_{j\in J}C^r_{K_j}(M,E)\to C^r_c(M,E)$,
$(\gamma_j)_{j\in J}\mto \sum_{j\in J}\gamma_j$ be the summation map,
which is linear, and continuous because it is continuous
on each summand. Then $S\circ \Phi=\id_{C^r_c(M,E)}$.
Hence~$\Phi$ has a continuous left inverse and hence~$\Phi$
is a topological embedding.\,\vspace{2mm}\Punkt

\noindent
{\bf Proof of Lemma~\ref{SJ}.}
Because $\Supp(\gamma|_S)\sub S\cap\Supp(\gamma)$
is compact for each $\gamma\in C^r_c(M,E)$,
the map~$\Phi$ makes sense,
and it is clear that~$\Phi$ is linear.

$\Phi$ is continuous:
If $K\sub M$ is compact, there is a finite set $F\sub P$
such that $K\sub \bigcup_{S\in F}S$. Then $\Phi(C^r_K(M,E))\sub \bigoplus_{S\in F} C^r_c(S,E)$,
whence the restriction $\Phi_K$ of $\Phi$ to $C^r_K(M,E)$ will
be continuous if we can show that all of its components
with values in $C^r_c(S,E)$ are continuous, for $S\in F$.
But these are the restriction maps $C^r_K(M,E)\to C^r_c(S,E)$, $\gamma\mto\gamma|_S$,
and hence are continuous because they can be written as a composition of the
continuous restriction map $C^r_K(M,E)\to C^r_{K\cap S}(S,E)$ (compare \cite[Lemma~3.7]{GCX})
and the continuous inclusion map $C^r_{K\cap S}(S,E)\to C^r_c(S,E)$.

If $P$ is a partition of~$M$ into open sets,
let $\Psi\colon \!\bigoplus_{S\in P}C^r_c(S,E)\!\to\! C^r_c(M,E)$
be the map taking a element $\eta:=(\gamma_S)_{S\in P}$ of the left hand side
to the function $\gamma\in C^r_c(M,E)$
defined piecewise via $\gamma(x):=\gamma_S(x)$
for $x\in S$. Then $\Phi(\Psi(\eta))=\eta$, and thus $\Phi$ is surjective.
Moreover, $\Psi(\Phi(\gamma))=\gamma$ for $\gamma\in C^r_c(M,E)$,
whence $\Phi$ is injective.
Hence $\Phi$ is bijective, with $\Psi=\Phi^{-1}$.
By the universal property of the locally convex direct sum,
the linear map $\Psi$ will be continuous
if its restriction $\Psi_S$ to the summand $C^r_c(S,E)$
is continuous, for each $S\in P$.
To check this property,
it suffices to show that the restriction $\Psi_K$ of $\Psi_S$ to
$C^r_K(S,E)$ is continuous for each compact set $K\sub S$.
But $\Psi_K$ is continuous, as it is the composition of the map $C^r_K(S,E)\to C^r_K(M,E)$
extending functions by~$0$ off $S$ (which is known to be continuous)\footnote{See
\cite[Lemma~4.24]{ZOO} if $r>0$; the case $r=0$ is elementary.}
and the continuous inclusion map $C^r_K(M,E)\to C^r_c(M,E)$.\,\vspace{2mm}\Punkt

\noindent
{\bf Proof of Lemma~\ref{framediff}.}
As the $C^{k+\ell}$-property can be tested on the open cover of chart domains,
we may assume that $M\sub \R^m$ is open.
The proof is by induction on~$k$.
If $k=0$, then $\gamma$ is $C^\ell$ by hypothesis (and $\ell=k+\ell$).
Now assume $k>0$. Then $\gamma$ is~$C^1$.
For each $X\in \cF_1$,
the map $X.\gamma$ is $C^{k-1}$
and $X_j\ldots X_1.X.\gamma$ is $C^\ell$
for all $j\in \N_0$ such that $j\leq k-1$
and $X_i\in \cF_{i+1}$ for $i\in \{1,\ldots, j\}$.
Hence $X.\gamma$ is $C^{k+\ell-1}$,
by induction. Let $\cF_1=\{X_1,\ldots, X_m\}$
and write $E_j=\frac{\partial}{\partial x_j}$
for $j\in \{1,\ldots, m\}$. Then $E_j=\sum_{i=1}^m a_{i,j}X_i$
with smooth functions $a_{i,j}\in C^\infty(M)$
and thus $\frac{\partial \gamma}{\partial x_j}=\sum_{i=1}^m a_{i,j}\, (X_i.\gamma)$
is $C^{k+\ell-1}$.
Since $\gamma$ is $C^1$ and its first order
partial derivatives are $C^{k+\ell-1}$,
the map $\gamma$ is $C^{k+\ell}$.\,\vspace{2mm}\Punkt

\noindent
{\bf Proof of Lemma~\ref{topframe}.}
Step~1. Let $\cU$ be the set of all open subsets
of~$M$ which are diffeomorphic to open subsets
of $\R^m$ (where $m$ is the dimension of~$M$).
For each $s\in \N_0\cup\{\infty\}$,
the topology on $C^s(M,E)$ is initial
with respect to the restriction
maps $\rho^s_U\colon C^s(M,E)\to C^s(U,E)$, $\gamma\mto\gamma|_U$
(see \cite[Lemma~4.12]{ZOO}).
Taking $s=r$, we deduce:\footnote{If the topology on a space~$X$ is initial
with respect to maps $f_i\colon X\to X_i$, with $i\in I$,
and the topology on~$X_i$ is initial with respect to maps $g_{j,i}\colon X_i\to X_{j,i}$
to topological spaces $X_{j,i}$, for $j\in J_i$,
then the topology on~$X$ is initial with respect to
the maps $g_{j,i}\circ f_i$.}
If the lemma holds for each space~$C^s(U,E)$,
then $\cO$ on $C^r(M,E)$ is initial with respect to the maps
\begin{equation}\label{leftrght}
D_{X_j|_U,\ldots, X_1|_U}\circ \rho^r_U=\rho^0_U\circ D_{X_j,\ldots,X_1}\,.
\end{equation}
Taking $s=0$, we deduce that $\cT_\cF$ is initial
with respect to the maps on the right hand side of~(\ref{leftrght}).
Hence $\cO=\cT_\cF$.

Step~2. In view of Step~1, it only remains to prove the lemma
assuming that there exists a $C^r$-diffeomorphism $\phi\colon M\to V$
onto an open set $V\sub \R^m$.
If $X$ is a smooth vector field on~$M$,
let us write $X':=T\phi\circ X\circ \phi^{-1}$ for the corresponding vector field
on~$V$.
Define $\Phi_s\colon C^s(M,E)\to C^s(V,E)$, $\gamma\mapsto \gamma\circ \phi^{-1}$
for $s\in \N_0\cup\{\infty\}$.
If $s\in \N\cup\{\infty\}$, then
$\Phi_{s-1}(X.\gamma)=X'.\Phi_s(\gamma)$ for each vector field~$X$ on~$M$ and $\gamma\in C^s(M,E)$,
whence
\[
\Phi_0\circ D_{X_j,\ldots, X_1}=D_{X_j',\ldots, X_1'}\circ \Phi_r
\]
for all $j\in \{0,\ldots, r\}$ and $X_i\in \cF_i$ for $i\in \{1,\ldots, j\}$.
Hence, since $\Phi_0$ and $\Phi_r$ are isomorphisms of topological vector spaces
(cf.\ \cite[Lemma~4.11]{ZOO}),
the topology on $C^r(M,E)$ is initial with respect to the $D_{X_j,\ldots, X_1}$
if and only if the topology on $C^r(V,E)$ is initial with respect
to the $D_{X_j',\ldots, X_1'}$.

Step~3. By Step~2, we may assume
that $M=V$ is an open subset of~$\R^m$.
We claim: If also $\cG=(\cG_1,\ldots,\cG_r)$ is an $r$-tuple of frames on~$V$,
then $\cT_\cG\sub \cT_\cF$.
Hence also $\cT_\cF\sub \cT_\cG$
(reversing the roles of $\cF$ and $\cG$)
and thus  $\cT_\cF = \cT_\cG$.
But it is known that $\cO=\cT_\cG$
if we choose $\cG_i:=\{\frac{\partial}{\partial x_1},\ldots,\frac{\partial}{\partial x_m}\}$
for all $i\in \{1,\ldots,r\}$
(cf.\ \cite[Proposition~4.4]{GCX}).
Thus $\cT_\cF=\cT_\cG=\cO$, as required.\\[2.5mm]
To establish the claim, recall that the multiplication operators
\[
m_f\colon C^0(V,E)\to C^0(V,E)\, , \quad m_f(\gamma):=f \cdot \gamma
\]
are continuous for each $f\in C^0(V)$
(cf.\ \cite[Lemma~3.9]{GCX}).
Write $\cF_i=\{Y_{i,1},\ldots, Y_{i,m}\}$.
Then each $X\in \cG_i$ is a linear combination
$X=\sum_{k=1}^m a_k Y_{i,k}$ 
with coefficients $a_k\in C^\infty(V)$.
Hence, for all $j\in \{0,\ldots, r\}$ and $X_1\in \cG_1$, $\ldots$, $X_j\in \cG_j$,
it follows from the product rule that $D_{X_j,\ldots,X_1}$
can be\linebreak
written as a sum of operators of the form $m_{f_{k_i,\ldots, k_1}} \circ D_{Y_{i,k_i}\ldots, Y_{1,k_1}}$,
where $i\in \{0,\ldots, j\}$, $k_1,\ldots, k_i\in \{1,\ldots, m\}$
and $f_{k_i,\ldots, k_1}\in C^\infty(V)$.
Since $\cT_\cF$ makes the map $D_{Y_{i,k_i}\ldots, Y_{1,k_1}}\colon C^r(V,E)\to C^0(V,E)$
continuous and also the multiplication operator $m_{f_{k_i,\ldots, k_1}}\colon C^0(V,E)\to C^0(V,E)$ is continuous,
it follows that $\cT_\cF$ makes $D_{X_j,\ldots,X_1}\colon C^r(V,E)\to C^0(V,E)$
continuous. Hence $\cT_\cG\sub \cT_\cF$.\,\vspace{2mm}\Punkt

\noindent
{\bf Proof of Lemma~\ref{enoughsn}.}
By definition, the topology on $C^\infty_K(G,E)$
is initial with respect to the inclusion maps $C^\infty_K(G,E)\to C^n_K(G,E)$
with $n\in \N_0$.
It hence suffices to prove the lemma for $t\in \N_0$.
For (a), let $\cF_i:=\cF_L$
for $i\in \{1,\ldots, t\}$
(with notation from Definition~\ref{deflrmrtr}).
For the proof of (b), let $\cF_i=\cF_R$
for $i\in \{1,\ldots, t\}$.
In either case, let
$\cF:=(\cF_1,\ldots,\cF_t)$.
Because the topology on $C^0_K(G,E)$
is defined by the seminorms $\|.\|_{p,\infty}$,
it follows with Lemma~\ref{topframe}
that the topology on $C^t_K(G,E)$ is defined
by the seminorms $\gamma\mto\|X_j\ldots X_1.\gamma\|_{p,\infty}$
with $j\in \{0,\ldots, t\}$ and $X_i\in \cF_i$ for $i\in \{1,\ldots, j\}$.
The pointwise maximum of these (for fixed $p$) is $\|.\|^L_{t,p}$ in (a),
$\|.\|^R_{t,p}$ in (b), from which the descriptions in (a)
and (b) follow.
We now prove the first half of the final assertion (the second half
can be shown analogously).
If $i\in \{0,\ldots, \ell\}$ and $j\in \{0,\ldots, k\}$
are given, let
$\cF=(\cF_1,\ldots, \cF_t)$ be the $t$-tupel
whose first $i$~entries are~$\cF_R$,
followed by $j$~entries $\cF_L$,
followed by $t-i-j$ arbitrary frames.
Then Lemma~\ref{topframe}
implies that $\gamma\mto \|X_{i+j}\ldots X_1.\gamma\|_{\infty,p}$
is continuous on $C^t_K(G,E)$,
for all $X_1\in \cF_1$, $\ldots$, $X_{i+j}\in \cF_{i+j}$.
The maximum of all these seminorms
for $i\leq \ell$ and $j\leq k$ is $\|.\|^{L,R}_{k,\ell,p}$,
which therefore is continuous.
Hence, the topology defined by these seminorms
is coarser than the given topology.
On the other hand, taking $\cF_i:=\cF_R$
for $i\in\{1,\ldots, \ell\}$ and $\cF_i=\cF_L$
for $i\in \{\ell+1,\ldots, t\}$,
Lemma~\ref{topframe}
shows that the topology on $C^t_K(G,E)$
is defined by the seminorms $\gamma\mto \|X_j\ldots X_1.\gamma\|_{p,\infty}$,
for $j\in \{0,\ldots, t\}$, continuous seminorms~$p$ on~$E$ and $X_i\in \cF_i$.
For fixed~$p$, each of the latter seminorms is bounded by
$\|.\|^{L,R}_{k,\ell,p}$.
Hence the topology defined by the $\|.\|^{L,R}_{k,\ell,p}$
is also finer than the given topology, and thus coincides with it.\,\vspace{2mm}\Punkt

\noindent
{\bf Proof of Lemma~\ref{transla}.}
We discuss $\tau^L_g$ ($\tau^R_g$ can be treated
analogously).
Since left translation $L_g\colon G\to G$, $L_g(x):=gx$
is a $C^r$-diffeomorphism, the map
$\Xi_g\colon C^r(G,E)\to C^r(G,E)$, $\gamma\mto \tau^L_g(\gamma)=\gamma\circ L_g$
is continuous and linear \cite[Lemma~3.7]{GCX}.
Hence also its restriction
$\Xi_{g,K}\colon C^r_K(G,E)\to C^r_{g^{-1}K}(G,E)$
is continuous,
and so is the map $\Xi_{g,c}\colon C^r_c(G,E)\to C^r_c(G,E)$, $\gamma\mto\tau^L_g(\gamma)$
(as it is linear and its restriction $\Xi_{g,K}$ to $C^r_K(G,E)$
is continuous for each $K$).
It is clear that each of the preceding maps is bijective; the
inverse map is given by $\Xi_{g^{-1}}$, $\Xi_{g^{-1}, g^{-1}K}$
and $\Xi_{g^{-1},c}$, respectively, and hence continuous.\,\vspace{2mm}\Punkt

\noindent
{\bf Proof of Lemma~\ref{normtrans}.}
If $X$ is a left invariant vector field on~$G$ and $\gamma\in C^1(G)$,
then $(X.(\tau^L_g\gamma))(a)= d(\gamma\circ L_g)(X(a))=d\gamma T(L_g)(X(a))=d\gamma X(ga)
=(X.\gamma)(ga)$ for $a\in G$ and thus $X.(\tau^L_g(\gamma))=\tau^L_g(X.\gamma)$.
Hence $\|X_j\ldots X_1.(\tau^L_g(\gamma))|_{g^{-1}K}\|_\infty
=\|\tau^L_g(X_j\ldots X_1.\gamma)|_{g^{-1}K}\|_\infty=
\|X_j\ldots X_1.\gamma|_K\|_\infty$ for all $j\in \{0,\ldots,\ell\}$
and $X_1,\ldots, X_j\in \cF_L$ (using the notation from\linebreak
Definition~\ref{deflrmrtr}).
Now take the maximum over all $j$ and $X_1,\ldots, X_j$.\,\vspace{2mm}\Punkt

\noindent
{\bf Proof of Lemma~\ref{invoaut}.}
The map $\eta_G\colon G\to G$, $x\mto x^{-1}$
is $C^r$. Hence
$(\eta_G)^*\colon
C^r(G,E)\to C^r(G,E)$,
$\gamma\mto \gamma\circ\eta_G$ is continuous linear \cite[Lemma~3.7]{GCX}.
As $f\colon G\to \, ]0,\infty[$, $f(x):=\Delta_G(x^{-1})$ is $C^r$,
we can consider the
multiplication operator $m_f\colon C^r(G,E)\to C^r(G,E)$, $m_f(\gamma)(x):=f(x)\gamma(x)$,
which is continuous linear (cf.\ \cite[Proposition~4.16]{ZOO} if $r>0$,
and \cite[Lemma~3.9]{GCX} otherwise).
Thus $\Theta=m_f\circ (\eta_G)^*$ is continuous linear.
Because $\Phi\circ\Phi=\id$, we deduce that $\Phi$ is an isomorphism of
topological vector spaces.
As a restriction of~$\Theta$, also the bijection $\Theta_K$
(with inverse $\Theta_{K^{-1}}$) is an isomorphism of topological vector spaces.
Finally, the linear map $\Phi_c$ is continuous
(as its restricions $\Phi_K$ to the spaces $C^r_K(G,E)$
are continuous) and hence an isomorphism
of topological vector spaces (as $\Phi_c\circ \Phi_c=\id$).\,\vspace{2mm}\Punkt

\noindent
{\bf Proof of Lemma~\ref{Phiv}.}
Since $\Phi_v$ is linear, it will be continuous
on $C^r_c(M)=\dl\,C^r_K(M)$\vspace{-.3mm}
if its restriction $\Phi_K\colon C^r_K(M)\to C^r_K(M,E)\sub C^r_c(M,E)$, $\gamma\mto \gamma v$
to $C^r_K(M)$ is continuous for each compact subset $K\sub M$.
Let $\mu\colon \R \times E\to E$ be the scalar multiplication,
and $h\colon M\to \R$ be a smooth map such that $L:=\Supp(h)$ is compact
and  $h|_K=1$. Because $\mu$ is smooth and $\mu(0,0)=0$,
also the bilinear map
$C^r_L(M,\mu)\colon C^r_L(M)\times C^r_L(M,E)\isom
C^r_L(M,\R\times E)\to C^r_L(E)$,
\[
(\gamma,\eta)\mto \mu\circ (\gamma,\eta)=\gamma\eta
\]
is smooth and hence continuous \cite[Proposition~3.10]{GCX}.
Hence also $\Phi_K(\gamma)=\gamma v = h\gamma v=\mu(\gamma,hv)$ is continuous
in $\gamma$.
To complete the proof, pick $\lambda\in E'$ such that $\lambda(v)=1$.
Then $C^r_c(M,\lambda)\colon C^r_c(M,E)\to C^r_c(M)$, $\gamma\mto\lambda \circ \gamma$
is a continuous linear map (by
\cite[Lemma~3.3]{GCX} and the locally convex direct limit property),
and $C^r_c(M,\lambda)\circ \Phi_v=\id_{C^r_c(M)}$
because $\lambda\circ (\gamma v)=\gamma$.
Since $\Phi_v$ has a continuous left inverse, it is a topological embedding.\,\vspace{2mm}\Punkt

\noindent
{\bf Proof of Lemma~\ref{scavect}.}
We first observe that the map $\Theta\colon E\to C^r(M,E)$ taking $v\in E$ to the constant map
$\Theta(v)\colon M\to E$, $x\mto v$ is continuous.
In fact, the linear map $\Theta\colon E\to C^0(M,E)$ is continuous,
as
\[
(\forall v\in E) \quad \|\Theta(v)\|_{p,K}:=\|\Theta(v)|_K\|_{p,\infty} \leq p(v)
\]
for each continuous seminorm $p$ on $E$ and compact subset $K\sub M$.
Since $d^k(\Theta(v))=0$ for all $k\in \N$ with $k\leq r$,
we see that $\Theta$ is also continuous as a map to $C^r(M,E)$.
We now use that $C^r(M,E)$ is a topological $C^r(M)$-module under pointwise multiplication:
Scalar multiplication $\mu\colon \R\times E\to E$ being continuous bilinear and hence
smooth, also $C^r(M,\mu)\colon C^r(M)\times C^r(M,E)\isom C^r(M,\R\times E)\to C^r(M,E)$,
$(\gamma,\eta)\mto \mu\circ (\gamma,\eta)=:\gamma\eta$ is smooth (see
\cite[Proposition~4.16]{ZOO} if $r>0$, \cite[Lemma~3.9]{GCX} if $r=0$) and hence continuous.
Thus $\Psi_E(\gamma,v)=\gamma\Theta(v)=C^r(M,\mu)(\gamma,\Theta(v))$ is continuous
in $(\gamma,v)$.\,\vspace{2mm}\Punkt

\noindent
{\bf Proof of Lemma~\ref{towardsh}.}
(a) It is clear from the definition of the topology that
any $0$-neighbourhood $U\!\sub\! C^r_K(M,E)$
contains an intersection $U_1\cap\ldots \cap U_n$\linebreak
of $0$-neighbourhoods of the form $U_i:=\{\gamma\in C^r_K(M,E)\colon \|d^{\ell_i}\gamma\|_{p_i,K_i}<\ve_i\}$
with $n\in \N$, $\ve_i>0$, $\ell_i\in \N_0$ such that $\ell_i\leq r$,
a continuous seminorm~$p_i$ on~$E$ and a compact set $K_i\sub T^{\ell_i}M$.
Let $(\wt{E}_p,\|.\|_p)$ be the Banach space associated to the continuous seminorm
$p:=p_1+\cdots+p_n$ on~$E$, and $\lambda_p\colon E\to\wt{E}_p$ be the canonical map.
Then $V_i:=\{\gamma\in C^r_K(M,\wt{E}_p)\colon \|d^{\ell_i}\gamma\|_{\|.\|_p,K_i}<\ve_i\}$
is an open $0$-neighbourhood in $C^r_K(M,\wt{E}_p)$
and hence also $V:=V_1\cap\ldots\cap V_n$.
Since $C^r_K(M,\lambda_p)^{-1}(V)\sub U$,
the assertion follows.\footnote{$C^r_K(M,\lambda_p)\colon C^r_K(M,E)\to C^r_K(M,\wt{E}_p)$, $\gamma\mto\lambda_p\circ\gamma$ is also continuous \cite[Lemma~3.3]{GCX}.}

(b) Since~$M$ is $\sigma$-compact,
there exists a locally finite $C^r$-partition of unity $(h_j)_{j\in \N}$ on~$M$
such that each $h_j$ has compact support
$K_j:=\Supp(h_j)$ (take a partition of
unity subordinate to a relatively compact
open cover using
Theorem~3.3 and Corollary~3.4 in \cite[Chapter~II]{Lan}
if $r>0$, \cite[Theorem~5.1.9]{Eng} if $r=0$).
Then~$\Phi$ from Lemma~\ref{lcsum}
is a topological embedding. Thus, for each $0$--neighbourhood $U\sub C^r_c(M,E)$,
there exist $0$-neighbourhoods $U_j\sub C^r_{K_j}(M,E)$
such that $\Phi^{-1}(\bigoplus_{j\in \N} U_j)\sub U$.
As a consequence of~(a), each~$U_j$ contains a set
of the form $C^r_{K_j}(M,\mu_j)^{-1}(V_j)$
for some Banach space~$E_j$,
continuous linear map $\mu_j\colon E\to E_j$,
and $0$-neighbourhood~$V_j$ in $C^r_{K_j}(M,E_j)$.
Then $F:=\prod_{j\in \N}E_j$
is a Fr\'{e}chet space, and $\lambda:=(\mu_j)_{j\in \N}\colon E\to F$
is continuous linear.
Let $\pi_j\colon F\to E_j$ be the projection
onto the $j$-th component,
and $W_j:=C^r_{K_j}(M,\pi_j)^{-1}(V_j)$.
Because $\Psi\colon C^r_c(M,F)\to\bigoplus_{j\in \N}C^r_{K_j}(M,F)$,
$\gamma\mto(h_j \gamma)_{j\in \N}$ is continuous linear
(Lemma~\ref{lcsum}),
the set $P:=\Psi^{-1}(\bigoplus_{j\in\N} W_j)$
is a $0$-neighbourhood in $C^r_c(M,F)$,
and hence $Q:=C^r_c(M,\lambda)^{-1}(P)$
is a $0$-neighbourhood in $C^r_c(M,E)$
(using \cite[Lemma~4.11]{GCX}).
If $\gamma\in Q$,
then $h_j(\lambda\circ \gamma)\in W_j$
for each $j\in \N$,
and hence
$\pi_j\circ (h_j(\lambda\circ \gamma))\in V_j$.
Since
$\pi_j\circ (h_j(\lambda\circ \gamma))
=h_j(\mu_j\circ\gamma)=\mu_j\circ (h_j\gamma)$,
we deduce that $h_j\gamma\in U_j$
and thus $\gamma\in U$. Thus $Q\sub U$,
and the assertion follows.

(c) Let $B\sub C^r_c(M,E)$ be
bounded.
As $M$ is locally compact and paracompact,
it admits a partition~$P$ into $\sigma$-compact
open sets \cite[Theorem~5.1.27]{Eng}.
Let $\Phi\colon C^r_c(M,E)\to \bigoplus_{S\in P}C^r_c(S,E)$
be as in Lemma~\ref{SJ}.
Then $\Phi(B)$ is bounded and hence
$\Phi(B)\in\bigoplus_{S\in P_0}C^r_c(S,E)$
for a finite set $P_0\sub P$.
After replacing the $S\in P_0$
by their union, we may assume that $B\sub C^r_c(S,E)$
(as a consequence of Lemma~\ref{SJ}, $C^r_c(S,E)$ can be regarded
as a topological vector subspace of $C^r_c(M,E)$).
Hence, we may assume that~$M$
is $\sigma$-compact.
Let $K_1,K_2,\ldots$ be compact sets
such that $M=\bigcup_{n=1}^\infty K_n$ and each $K_n\sub K_{n+1}^0$.
Then $C^r_c(M,E)$ is the locally convex direct limit
of $C^r_{K_1}(M,E)\sub C^r_{K_2}(M,E)\sub\cdots$,
where $C^r_{K_n}(M,E)=\{\gamma\in C^r_{K_{n+1}}(M,E)\colon (\forall x\in M\setminus K_n)\;
\gamma(x)=0\}$
is a closed vector subspace of $C^r_{K_{n+1}}(M,E)$
and $C^r_{K_{n+1}}(M,E)$ induces the given topology of $C^r_{K_n}(M,E)$,
for $n\in \N$.
Hence~$B$ is a bounded set in
$C^r_{K_n}(M,E)$ for an $n\in \N$ \cite[II.6.5]{SaW}.\,\vspace{2mm}\Punkt

\noindent
{\bf Proof of Lemma~\ref{convcts}.}
If $(\gamma*_b\eta)(x)\not=0$, then by (\ref{suppgand}) there is $y\in \Supp(\eta)$ such that
$xy^{-1}\in \Supp(\gamma)$. Hence $x\in \Supp(\gamma)y\sub \Supp(\gamma)\Supp(\eta)$.
Now assume that~$K$ is compact.
Because the integrand in (\ref{goodpar}) is continuous as a map taking $(x,y)\in G\times \Supp(\gamma)$ to~$F$,
the continuity of $\gamma*_b\eta$ follows from Lemma~\ref{ctspar}.
If $M\sub G$ is compact and~$q$ a continuous seminorm on~$F$, there are continuous seminorms~$p_1$ on~$E_1$
and~$p_2$ on~$E_2$ such that $q(b(v,w))\leq p_1(v)p_2(w)$ for all $v\in E_1$, $w\in E_2$.
For all $x\in M$, we infer
$q((\gamma*_b\eta)(x)) \leq \int_Kp_1(\gamma(y))p_2(\eta(y^{-1}x))\,d\lambda_G(y)
\leq \lambda_G(K)\|\gamma\|_{p_1,\infty}\|\eta|_{K^{-1}M}\|_{p_2,\infty}$. Thus
\begin{equation}\label{prepa1}
\|(\gamma*_b\eta)|_M\|_{q,\infty}\leq \lambda_G(K)\|\gamma\|_{p_1,\infty}\|\eta|_{K^{-1}M}\|_{p_2,\infty},
\end{equation}
and hence~$\beta$ is continuous.\\[2.3mm]
If~$L$ is compact, we have $(\gamma*_b\eta)(x)=\int_L
b(\gamma(xz^{-1}),\Delta_G(z^{-1})\eta(z))\, d\lambda_G(z)$ by (\ref{altform}),
from which continuity of $\gamma*_b\eta$ follows.
Finally, $\beta$ is continuous as
$\|(\gamma*_b\eta)|_M\|_{q,\infty}\leq \lambda_G(L)
\|\gamma|_{ML^{-1}}\|_{p_1,\infty}\|\eta\|_{p_2,\infty}\|\Delta_G|_{L^{-1}}\|_\infty$.\vspace{2mm}\Punkt

\noindent
{\bf Proof of Proposition~\ref{diffbar}.}
We may assume $r,s\in \N_0$
and proceed by induction,
starting with $r=0$.
If also $s=0$, see Lemma~\ref{convcts}.

Now let $s>0$. If $x_0\in G$, let $V\sub G$ be an open neighbourhood of $x_0$
with compact closure $\wb{V}$. If $K$ is compact, set $M:=K$.
If $K$ is not compact, then $M:=\wb{V}L^{-1}$ is compact.
In either case, $(\gamma*_b\eta)(x)=
\int_Mb(\gamma(y),\eta(y^{-1}x))\,d\lambda_G(y)$
for all $x\in V$.
Hence
$(\gamma*_b\eta)|_V$ is $C^s$,
by Lemma~\ref{diffpar},
and thus $\gamma*_b\eta$ is $C^s$.
The lemma also entails that
\begin{eqnarray*}
\cL_{v_1}.(\gamma*_b\eta)(x) & = &
\int_Mb(\gamma(y), d\eta(T(L_{y^{-1}})(\cL_{v_1}(x)))\,d\lambda_G(y)\\
& = & \int_Gb(\gamma(y), d\eta(T(L_{y^{-1}})T(L_x)(v_1))
\,d\lambda_G(y)\\
& = & \int_Gb(\gamma(y), (\cL_{v_1}.\eta)(y^{-1}x))\,d\lambda_G(y)
\,=\,  (\gamma*_b(\cL_{v_1}.\eta))(x)\, ,
\end{eqnarray*}
using that $T(L_{y^{-1}})T(L_x)(v_1)=T(L_{y^{-1}}\circ L_x)(v_1)=T(L_{y^{-1}x})(v_1)=\cL_{v_1}(y^{-1}x)$.
Since $\cL_{v_1}\eta\in C^{s-1}_L(G,E_2)$,
we obtain, by induction on~$s$,
\begin{equation}\label{caseone}
\cL_{v_i}\cdots\cL_{v_1}.(\gamma*_b\eta)
=
\cL_{v_i}\cdots\cL_{v_2}.(\gamma*_b\cL_{v_1}.\eta)
=\gamma*_b(\cL_{v_i}\cdots\cL_{v_1}.\eta).
\end{equation}
Now assume that $r>0$.
If $s=0$, then (\ref{altform})
enables Lemma~\ref{diffpar} to be applied,\footnote{For $x\in V$ as above,
we can replace the domain of integration by a compact set
again.}
and thus $\gamma*_b\eta\in C^r(G,F)$.
Moreover, repeating the arguments leading to (\ref{caseone})
with right translations, we deduce from (\ref{altform})
that
\begin{equation}\label{casetwo}
\cR_{w_j}\cdots\cR_{w_1}.(\gamma*_b\eta)
=
(\cR_{w_j}\cdots\cR_{w_1}.\gamma)*_b\eta\, .
\end{equation}
If $\gamma$ is $C^r$ and $\eta$ is $C^s$,
then
$\cL_{v_i}\cdots\cL_{v_1}.(\gamma*_b\eta)
=
\gamma*_b (\cL_{v_i}\cdots\cL_{v_1}.\eta)$ is $C^r$ by the case $s=0$,
and (\ref{hinein}) holds.
Thus $\gamma*_b\eta$ is $C^{r+s}$, by Lemma~\ref{framediff}.
In view of Lemmas \ref{topframe} and~\ref{convcts},
the right hand side of~(\ref{hinein})
is continuous as a map $C^r_K(G,E_1)\times C^s_L(G,E_2)\to C^0_{KL}(G,F)$,
for all $v_1,\ldots, v_i$ and $w_1,\ldots,w_j$.
Hence~$\beta$ is continuous as a map to $C^{r+s}_{KL}(G,F)$,
by Lemma~\ref{topframe}.\vspace{2mm}\Punkt

\noindent
{\bf Proof of Lemma~\ref{invo}.}
Substituting $z=xy$ and using the left invariance of Haar measure, we obtain
$(\gamma*_b\eta)^*(x) =
\Delta_G(x^{-1})(\gamma*_b\eta)(x^{-1})=$\linebreak
$\Delta_G(x^{-1})\int_G b(\gamma(y),\eta(y^{-1}x^{-1})\, d\lambda_G(y)=\Delta_G(x^{-1})\int_G b(\gamma(x^{-1}z),\eta(z^{-1}))\, d\lambda_G(z)$\linebreak
$=\int_G b^\vee(\eta^*(z),\gamma^*(z^{-1}x))\, d\lambda_G(z) =
(\eta^* *_{b^\vee}\gamma^*)(x)$.\vspace{2mm}\Punkt

\noindent
{\bf Proof of Lemma~\ref{efftra}.}
(a) With $z=g^{-1}y$ and left invariance of $\lambda_G$, we get
\begin{eqnarray*}
(\tau_g^L(\gamma*_b\eta))(x) & = & (\gamma*_b\eta)(gx)
=\int_Gb(\gamma(y),\eta(y^{-1}gx))\,d\lambda_G(y)\\
&=&\int_Gb(\gamma(gz),\eta(z^{-1}x))\,d\lambda_G(z)
=((\tau^L_g\gamma)*_b\eta)(x).
\end{eqnarray*}
(b) For $x\in G$, get
$\tau^R_g(\gamma *_b\eta)(x)=(\gamma*_b\eta)(xg)=\int_Gb(\gamma(y),\eta(y^{-1}xg))\,d\lambda_G(y)$
$=\int_Gb(\gamma(y),\tau^R_g(\eta)(y^{-1}x))\,d\lambda_G(y)
=(\gamma*_b\tau^R_g(\eta))(x)$.\vspace{2mm}\Punkt

\noindent
{\bf Proof of Lemma~\ref{lasimpe}.}
Let $\cF_R$ and $\cF_L$ be as in Definition~\ref{deflrmrtr}.
If $i\in \{0,\ldots, k\}$, $j\in \{0,\ldots,\ell\}$,
$X_1,\ldots,X_i\in \cF_R$
and $Y_1,\ldots, Y_j\in \cF_L$, then
\begin{eqnarray*}
\lefteqn{\|X_1\ldots X_i Y_1\ldots Y_j.(\gamma*_b\eta)\|_{q,\infty}
\,=\, \|(X_1\ldots X_i.\gamma)*_b(Y_1\ldots Y_j.\eta)\|_{q,\infty}}\\
&\leq& \|X_1\ldots X_i.\gamma\|_{p_1,\infty}
\|Y_1\ldots Y_j.\eta\|_{p_2,\infty}\lambda_G(K)\,\leq\,   \|\gamma\|^R_{k,p_1}\|\eta\|^L_{\ell,p_2}\lambda_G(K)
\end{eqnarray*}
by (\ref{hinein}) and (\ref{prepa1}).
All assertions now follow by passage to maxima
over suitable $i,j$ and the corresponding vector fields.\vspace{2mm}\Punkt

\noindent
{\bf Proof of Proposition~\ref{bihyp}.}
Let us write $\beta_b$ for $\beta$.

Step 1: Assume that $G$ is $\sigma$-compact.
Then the topology on $C^t_c(G,F)$ is initial with respect to
maps of the form $C^t_c(G,\lambda_i)$ for certain continuous linear
maps $\lambda_i\colon F\to F_i$ to Fr\'{e}chet spaces (Lemma~\ref{towardsh}\,(b)).
Hence, by Lemma~\ref{basichypo}\,(c),
$\beta_b$ will be hypocontinuous if we can show that $C^t_c(G,\lambda_i)\circ \beta_b=\beta_{\lambda_i\circ b}$
is hypocontinuous for all $i \in I$. Thus, we may assume that $F$ is a Fr\'{e}chet space.
Then $b=c\circ (\psi_1\times \psi_2)$ with certain continuous linear maps $\psi_1\colon E_1\to F_1$\linebreak
and $\psi_2\colon E_2\to F_2$ to Fr\'{e}chet spaces and a continuous bilinear map\linebreak
$c\colon F_1\times F_2\to F$
(see Lemma~\ref{basichypo}\,(e)). Thus
$\beta_b=\beta_c\circ (C^r_c(G,\psi_1)\times C^s_c(G,\psi_2))$,
and we need only show that $\beta_c$ is hypocontinuous (Lemma~\ref{basichypo}\,(b)).
Hence $E_1$ and $E_2$ are Fr\'{e}chet spaces, without loss of generality.
Then $C^r_c(G,E_1)$ and $C^s_c(G,E_2)$ are locally convex direct limits of Fr\'{e}chet spaces
and hence barrelled \cite[II.7.1 and II.7.2]{SaW},
whence $\beta$ will be hypocontinuous if we can show that it is separately continuous
(by~\cite[III.5.2]{SaW}). For fixed $\eta\in C^s_c(G,E_2)$, let $L:=\Supp(\eta)$.
The map $\beta(.,\eta)\colon C^r_c(G,E_2)\to C^t_c(G,F)$
being linear, it will be continuous on $C^r_c(G,E_1)=\dl\,C^r_K(G,E_1)$\vspace{-.3mm}
if we can show that its restriction to $C^r_K(G,E_1)$
is continuous for each compact set $K\sub G$.
But this is the case, since the convolution map
$C^r_K(G,E_1)\times C^s_L(G,E_2)\to C^t_{KL}(G,F)\sub C^t_c(G,F)$
is continuous, by Lemmas~\ref{convcts} and~\ref{diffbar}.
By an analogous argument, $\beta(\gamma,.)\colon C^s_c(G,E_2)\to C^t_c(G,F)$ is continuous
for each $\gamma\in C^r_c(G,E_1)$.

Step 2.
Let $H\sub G$ be a $\sigma$-compact open subgroup,
$G/H:=\{gH\colon g\in G\}$
be the set of left cosets and $H\backslash G:=\{Hg\colon g\in G\}$
the set of right cosets.
Since $G/H$ is a partition of~$G$ into open sets,
we can identify $C^r_c(G,E_1)$ with $\bigoplus_{M\in G/H}C^r_c(M,E_1)$,
by Lemma~\ref{SJ}. In particular, we can regard $C^r_c(M,E_1)$ as a topological vector subspace
of $C^r_c(G,E_1)$ (extending functions by~$0$).
Likewise, $C^s_c(G,E_2)$ can be identified with $\bigoplus_{N\in H\backslash G}C^s_c(N,E_2)$.
By Lemma~\ref{basichypo}\,(d),
$\beta$ will be hypocontinuous if we can show that its restriction
to $\beta_{M,N}\colon C^r_c(M,E_1)\times C^s_c(N,E_2)\to C^t_c(G,F)$
is hypocontinuous for all $M\in G/H$ and $N\in H\backslash G$.
Write $M=mH$ and $N=Hn$ with suitable $m,n\in G$.
Using Lemma~\ref{efftra}, we can write
\begin{equation}\label{manytr}
\beta_{M,N}=\tau^L_{m^{-1}}\circ \tau^R_{n^{-1}}\circ \beta_{H,H}\circ (\tau^L_m \times \tau^R_n),
\end{equation}
where
$\tau^L_m\colon C^r_c(M,E_1)\to C^r_c(H,E_1)$,
$\tau^R_n\colon C^s_c(N,E_2)\to C^s_c(H,E_2)$,\linebreak
$\tau^R_{n^{-1}}\colon C^t_c(H,F)\to C^t_c(N,F)$ and
$\tau^L_{m^{-1}}\colon C^t_c(N,F)\to C^t_c(mN,F)\sub C^t_c(G,F)$
are the respective translation maps, which are continuous as restrictions of translation
maps on spaces of test functions on~$G$
(as in Lemma~\ref{transla}).
Since $\beta_{H,H}\colon C^r_c(H,E_1)\times C^s_c(H,E_2)\to C^t_c(G,F)$
is hypocontinuous by Step~1, using Lemma~\ref{basichypo} (a) and (b)
we deduce from (\ref{manytr})
that also each of the maps $\beta_{M,N}$ is hypocontinuous. This completes the proof.\vspace{2mm}\Punkt

\noindent
{\bf Proof of Lemma~\ref{summeas}.}
As $\Phi$ is linear, it will be continuous if its restriction $\Phi_K$ to $\Mea_K(X)$
is continuous for each compact set $K\sub X$.
By the hypotheses, $\Phi(\Mea_K(X))$ is contained in the finite
direct sum $\bigoplus_{j\in J_K} \Mea(A_j)$,
whence $\Phi_K$ will
be continuous if we can show that its components
with values in $\Mea(A_j)$ are continuous, for all $j\in J_K$.
But these are continuous, as they have
operator norm $\leq 1$
(noting that $\|\mu|_{\cB(A_j)}\|=|\mu|(A_j)\leq|\mu|(X)=\|\mu\|$).\vspace{2mm}\Punkt

\noindent
{\bf Proof of Lemma~\ref{m-emb}.}
$\Phi$ is continuous:
If $K\sub X$ is compact, there is $n\in \N$
such that $K\sub K_n$. Then $\Phi(\Mea_K(X))\sub \bigoplus_{j\leq n} \Mea_{K_j}(X)$,
whence the restriction $\Phi_K$ of $\Phi$ to $\Mea_K(X)$ will
be continuous if we can show that all of its components
with values in $\Mea_{K_j}(X)$ are continuous, for $j\in \{1,\ldots, n\}$.
But these take $\mu\in \Mea_K(X)$ to $\one_{K_j\setminus K_{j-1}}\odot\mu$,
and hence are continuous, as they have
operator norm $\leq 1$
(since $\|\one_{K_j\setminus K_{j-1}}\odot\mu\|=|\one_{K_j\setminus K_{j-1}}\odot\mu|(X)
= (\one_{K_j\setminus K_{j-1}}\odot |\mu|)(X)
=|\mu|(K_j\setminus K_{j-1})\leq|\mu|(X)=\|\mu\|$.)
Now consider the map $S\colon \bigoplus_{n\in N}\Mea_{K_n}(X)\to \Mea_c(X)$,
$(\mu_n)_{n\in \N}\mto\sum_{n=1}^\infty\mu_n$,
which is continuous as it is continuous on each summand
and linear. Then $S\circ\Phi=\id_{\Mea_c(X)}$.
Thus~$\Phi$ has a continuous left inverse, and hence~$\Phi$ is a topological embedding.\vspace{2mm}\Punkt

\noindent
{\bf Proof of Lemma~\ref{intomas}.}
The linear map $\Phi$ will be continuous if its restriction
$\Phi_K$ to $C_K(X)$ is continuous for each compact set $K\sub X$.
The latter holds, since
$\|\Phi_K(\gamma)\|=\|\gamma\odot \mu\|=(|\gamma|\odot \mu)(X)
=\|\gamma\|_{L^1}\leq \mu(K)\|\gamma\|_\infty$.
Likewise, the restriction $\Psi_K$ of $\Psi$ to $L^1_K(X)$ is continuous because
$\|\Psi_K(\gamma)\|\leq \|\gamma\|_{L^1}$.\vspace{-2mm}\Punkt
{\small
Corresponding author:\\[3mm]
Helge  Gl\"{o}ckner, Universit\"at Paderborn, Institut f\"{u}r Mathematik,\\
Warburger Str.\ 100, 33098 Paderborn, Germany;\\[2mm]
Email: {\tt  glockner\at{}math.upb.de}}
\end{document}